\documentclass[twoside]{article}
\usepackage{arxiv}
\pdfoutput=1
\usepackage[utf8]{inputenc} % allow utf-8 input
\usepackage[T1]{fontenc}    % use 8-bit T1 fonts
\usepackage{hyperref}       % hyperlinks
\usepackage{url}            % simple URL typesetting
\usepackage{booktabs}       % professional-quality tables
\usepackage{amsfonts}       % blackboard math symbols
\usepackage{nicefrac}       % compact symbols for 1/2, etc.
\usepackage{microtype}      % microtypography
\usepackage{verbatim}
\usepackage{graphicx}
\usepackage{mathtools}
\usepackage{amssymb,amsmath,color}
\usepackage{url}
\usepackage{mathrsfs}
\usepackage{dsfont}
\usepackage[mathcal]{eucal}
\usepackage{empheq}
\usepackage{algorithmic}
\usepackage{algorithm}

\usepackage{symbol_definitions}

\newcommand{\BEAS}{\begin{eqnarray*}}
\newcommand{\EEAS}{\end{eqnarray*}}
\newcommand{\BEA}{\begin{eqnarray}}
\newcommand{\EEA}{\end{eqnarray}}
\newcommand{\BEQ}{\begin{equation}}
\newcommand{\EEQ}{\end{equation}}
\newcommand{\BIT}{\begin{itemize}}
\newcommand{\EIT}{\end{itemize}}
\newcommand{\BNUM}{\begin{enumerate}}
\newcommand{\ENUM}{\end{enumerate}}
\newcommand{\BA}{\begin{array}}
\newcommand{\EA}{\end{array}}

\newcommand{\argmin}{\mathop{\rm argmin}}
\newcommand{\argmax}{\mathop{\rm argmax}}

\newcommand{\ri}{\mathop{\rm ri}}

\newcommand{\BlackBox}{\rule{1.5ex}{1.5ex}}  % end of proof
\newenvironment{proof}{\par\noindent{\bf Proof\ }}{\hfill\BlackBox\\[2mm]}
 
\newtheorem{theorem}{Theorem}
\newtheorem{lemma}[theorem]{Lemma} 
\newtheorem{proposition}[theorem]{Proposition}

\newtheorem{definition}[theorem]{Definition}

\def\eps{\varepsilon}
\def\one{\mathbf{1}}
\def\t{^{\top}}

\def\dfn{{\,\triangleq\,}}
\def\pr{\sP\!}
\def\lip{{L\!}}
\def\prox{{\mathsf{Prox}}}
\def\grad{{\nabla}}

\def\ind{\mathsf{1}}
\def\supp{{\sigma\!}}

\def \SPS{\textrm{SPS}}
\def \MPICS{\textrm{SMP}}
\def \EPPD{\textrm{EPPD}}
\def \EPAPD{\textrm{EPAPD}}

\newcommand{\re}[1]{(\ref{eq:#1})}

\newcommand{\ip}[2]{{ \langle{#1},{#2}\rangle}}
\jmlrheading{}{2017}{}{}{}{}{Achintya Kundu, Francis Bach and Chiranjib Bhattacharyya}
\ShortHeadings{Convex optimization over intersection of simple sets}{Achintya Kundu, Francis Bach and Chiranjib Bhattacharyya}
\firstpageno{1}

\begin{document}
	
	\title{Convex optimization over intersection of simple sets: improved convergence rate guarantees via an exact penalty approach}
	
	\author{
		\name Achintya Kundu  \email achintya@csa.iisc.ernet.in \\
		\addr Department of Computer Science \& Automation \\ Indian Institute of Science, Bangalore, India       
		\AND  
		\name Francis Bach \email francis.bach@inria.fr \\
		\addr INRIA - Sierra Project-team \\ ~{\'E}cole Normale Sup{\'e}rieure, Paris, France
		\AND 
		\name Chiranjib Bhattacharrya \email chiru@csa.iisc.ernet.in \\
		\addr Department of Computer Science \& Automation \\ Indian Institute of Science, Bangalore, India       
	}
	
	\maketitle

\begin{abstract}
We consider the problem of minimizing a convex function over the intersection of finitely many simple sets which are  easy to project onto. This is an important problem arising in various domains such as machine learning. The main difficulty lies in finding the projection of a point in the intersection of many sets. Existing approaches yield an infeasible point with an iteration-complexity of $O(1/\varepsilon^2)$ for nonsmooth problems with no guarantees on the in-feasibility. By reformulating the problem through exact penalty functions, we derive first-order algorithms which not only guarantees that the distance to the intersection is small but also improve the complexity to $O(1/\varepsilon)$ and $O(1/\sqrt{\varepsilon})$ for smooth functions. For composite and smooth problems, this is achieved through a saddle-point reformulation where the proximal operators required by the primal-dual algorithms can be computed in closed form. We illustrate the benefits of our approach on a graph transduction problem and on graph matching.
\end{abstract}

\section{Introduction}
We call a closed convex set \emph{simple} if there is an oracle available for computing Euclidean projection onto the set. In this paper we consider the problem of minimizing a convex function $f$ over a convex set $\C$ where $\C$ is given as the intersection of finitely many simple closed convex sets $\C_1,\ldots,\C_m$ ($m \ge 2$). Specifically, we focus on optimization problems of the following form:
%\begin{empheq}[box=\fbox]{align}
\BEQ
\label{eq:p0} f_* ~=~ \min_{ \xx ~\in ~\X } ~\Big[~  f(\xx)  ~ + ~  \sum_{i=1}^m \ind_{\C_i}(\xx) ~\Big], 
\EEQ
%\end{empheq}
where $\ind_{\C_i}$ is the indicator function for set $\C_i$ and $\X$ $(\C \subset \X)$ represents the domain of $f$.

Optimization problems of the form \re{p0} arise in many machine learning tasks such as learning over doubly stochastic matrices, matrix completion \cite{DRS_matrix_completion},  graph transduction \cite{chiru2015nips}; sparse principal component analysis can be posed as optimization over the intersection of the set of positive semidefinite (PSD) matrices with unit trace and an $\ell_1$-norm ball \cite{Aspremont_Sparse_PCA}; in learning correlation matrices, the feasible set is the intersection of the PSD cone and the set of symmetric matrices with diagonal elements equal to one \cite{Nearest_Correlation_2016}. Another area of computer science where problems of type \re{p0} occur is in convex relaxations of various combinatorial optimization problems such as correlation clustering \cite{Correlation_Clustering}, graph-matching \cite{Graph_Matching_Bach}, etc. 

Over the last few decades a large number of first-order algorithms have been proposed to solve \re{p0} efficiently assuming $\C$ to be simple \cite{Nemirovski_MLopt1, Nemirovski_MLopt2, Nesterov_smoothing_nonsmooth}. But, in many practical problems such as those mentioned above, projection onto the feasible set $\C = \cap_{i=1}^m\C_i$ is difficult to compute whereas oracles for projecting onto each of $\C_1,\ldots,\C_m$ are readily available. Note that many sets~$\C$ where Frank-Wolfe algorithms can sometimes be used~\cite{jaggi2013revisiting}, i.e., when maximizing linear functions on $\C$ is supposed to be efficient, can often be decomposed as the intersection of sets with projection oracles (a classical example being the set of doubly stochastic matrices, as done in our experiments). 

This calls for developing efficient first-order algorithms which access $\C$ only through the projection oracles of the individual sets $\C_1, \ldots,\C_m$. We mention here that such algorithms have been well-studied in the context of two specific problems: (a) the convex feasibility problem (corresponding to $f = 0$), which aims at finding a point in $\C = \cap_{i=1}^m \C_i$ \cite{Bauschke1996, Beck_error_bound} and (b) the problem of computing Euclidean projections onto $\cap_{i=1}^m \C_i$ \cite{Dykstra_Projection, Beck_dual_proximal}. Existing algorithms for (a)  and (b) ensure a feasible solution only in the asymptotic sense and in general produce only an infeasible approximate solution when terminated after a finite number of iterations. Therefore, aiming for feasible approximate solution without access to projection oracle for $\C$ seems too big a goal to achieve for problems of the form~\re{p0}. Hence, we relax the feasibility requirement and introduce the following notion of approximate solution: 
\begin{definition}
For a given $\eps >0$, we call $\xx_{\eps} \in \X$ to be an $\eps$-optimal $\eps$-feasible solution of \re{p0} if $f(\xx_\eps) - f_* \le \eps$ and $d_\C(\xx_\eps) \le \eps / \lip_f$, where $d_\C(\xx_\eps) \dfn \inf_{\xx \in \C} \| \xx - \xx_\eps\|$ and $\lip_f$ is the Lipschitz constant of $f$.
\end{definition}
Note that $f(\xx_\eps) \ge f_*$ holds if $\xx_\eps$ is feasible. Since $\xx_\eps$ is allowed to be infeasible as per above definition, $f(\xx_\eps)$ might be well below $f_*$. The bound on the distance to feasible set $d_\C(\xx_\eps) \le \eps / \lip_f$ not only characterizes that feasibility violation of $\xx_\eps$ is small but also ensures $f(\xx_\eps) - f_* \ge -\eps$. With the notion of approximate solution in place, the key question now is the following: given access to projection oracles of the $\C_i$'s how many oracle calls does a first-order method need in order to produce an $\eps$-optimal $\eps$-feasible solution of~\re{p0}. In this paper we aim to address this question. We summarize our contributions below.

\paragraph{Contributions.}
To the best of our knowledge, we are the first one to derive general complexity results for problems of the form \re{p0} where $f$ is given by a first-order oracle and the feasible set $\C = \cap_{i=1}^m \C_i$ can be accessed only through projections onto $\C_i$s. Note that our complexity estimates not only guarantee closeness of the approximate solution to the optimal objective value but also provide guarantees on the distance of such infeasible solutions from the feasible set. More precisely (see summary in Table~\ref{table:complexity}): 
\BIT
\item Utilizing a standard constraint qualification assumption on problem \re{p0}, we present in Proposition~\ref{prop:exact_penalty} an exact penalty based reformulation whose $\eps$-optimal feasible solutions are in fact the desired $\eps$-accurate $\eps$-feasible solution of \re{p0}. 

\item We show in Proposition \ref{prop:nonsmooth_convex} that an adaptation of the standard subgradient method achieves the $O(1/ {\eps^2})$ iteration complexity for obtaining an $\eps$-optimal $\eps$-feasible solution of \re{p0} where $f$ belongs to the class of general nonsmooth convex functions given by a first-order oracle. Specifically, an iteration of the proposed algorithm asks for one call to the first-order oracle of $f$ and one call each to the projection oracles of $\C_1,\ldots,\C_m$. Additionally, assuming $f$ to be strongly convex we show in Proposition \ref{prop:nonsmooth_strong} that the same subgradient based algorithm achieves the $O(1/ {\eps})$ iteration complexity for obtaining an $\eps$-optimal $\eps$-feasible solution of~\re{p0}. We mention that existing approaches \cite{Bertsekas_Incremental} with  $O(1/\eps^2)$ complexity produce only an infeasible solution without any guarantee on the distance of the infeasible solutions from the feasible set. For the strongly convex case $O(1/\eps)$ complexity was reported \cite{Beck_dual_proximal} but applicable only to a limited class of functions $f$ where gradients of Fenchel conjugate of $f$ can be computed easily. In contrast, our method relies on the availability of only subgradient of $f$.   
 
\item Through a novel saddle-point reformulation and employing existing primal-dual methods we show that the resulting approach achieves $O(1/ {\eps})$ iteration complexity for obtaining an $\eps$-optimal $\eps$-feasible solution of \re{p0} when $f$ belongs to the class of smooth convex functions given by a first-order oracle. Similar to the subgradient approach, an iteration of the proposed primal-dual approach requires one call to the first-order oracle of $f$ and one call each to the projection oracles of $\C_1,\ldots,\C_m$. Further, assuming $f$ to be strongly convex the same primal-dual approach achieves $O(1/ \sqrt{\eps})$ iteration complexity for obtaining an $\eps$-optimal $\eps$-feasible solution of \re{p0}. Moreover, for nonsmooth convex functions with specific structure, for example, when the minimization problem has a smooth convex-concave saddle-point representation, we show that an adaptation of the mirror-prox technique achieves an iteration complexity of $O({1}/{\eps})$ to produce an $\eps$-optimal $\eps$-feasible solution of \re{p0}. For the same class of functions, existing approaches \cite{Nemirovski_mirror_prox_composite} using mirror-prox technique reported $O(\frac{1}{\eps}\log\frac{1}{\eps})$ complexity.

\EIT

\begin{table}
	\caption{Complexity of the proposed first-order algorithms for obtaining an $\eps$-optimal $\eps$-feasible solution of \re{p0} under 4 different classes of functions $f$.} \label{table:complexity}
	\begin{center}
		\begin{tabular}{|l||c|c|}
			\hline
			{\bf Class of functions $f$} 	&{\bf Nonsmooth}  	&{\bf Smooth} 	\\ \hline \hline
			{\bf Convex}         	&$O(1/\eps^2)$ 		&$O(1/\eps)$	\\ \hline
			{\bf Strongly convex}  	&$O(1/\eps)$ 		&$O(1/\sqrt{\eps})$ \\ \hline
		\end{tabular}
	\end{center}
\end{table}

\paragraph{Notation.} 
Through out this paper $\|\! \cdot \!\|$  denotes the standard Euclidean norm. Let $\A \subset \RR^n $ be a nonempty closed convex set and $\xx \in \RR^n$. Euclidean projection (or simply projection) of $\xx$ onto $\A$ is given by $\pr_{\A}(\xx) = \argmin_{ \aa \in \A} \| \xx - \aa \|$. The distance of $\xx$ from $\A$ is given by $d_\A(\xx)  \dfn \min_{ \aa \in \A} \| \xx - \aa \| = \|\xx - \pr_{\A}(\xx)\|$. The support function of $\A$ is defined as ~$\supp_\A(\xx) \dfn \sup_{\aa \in \A} \ip{\xx}{\aa},~  \xx \in \RR^n$. Let $\psi:\A \to \RR$ be convex; its proximal operator  is defined as $\prox_{\gamma \psi}(\xx) \dfn  \argmin_{ \aa \in \A } \left[ \psi(\aa) +  \frac{1}{2\gamma}\| \xx -\aa \|^2\right],\,\gamma>0$. Whenever $\frac{0}{0}$ appears we will treat it to be $0$. Proofs of all Propositions \& Lemmas and details of the proposed algorithms are given in the Appendix.

\section{Problem Set-up \& Related Work}
In this paper we focus on developing efficient first-order algorithms for solving problems of the form \re{p0}. For the rest of this paper, we make the following assumptions on \re{p0}:
\BNUM
\item[{\bf A1.}] $\X, \C_1, \, \ldots,\,\C_m$ are simple closed convex sets in $\RR^n$ such that $\X$ is bounded and contains $\C \dfn \bigcap_{i=1}^m \C_i$,
%\item[{\bf A2.}] the feasible set $\C \dfn \bigcap_{i=1}^m \C_i$ is a nonempty subset of the domain $\X$ and $\X$ is bounded,
\item[{\bf A2.}] $f: \X \to \RR$ is convex and Lipschitz-continuous with Lipschitz constant $\lip_f >0$,
\item[{\bf A3.}] the family $\{\C_1, \, \ldots,\,\C_m\}$ satisfies the standard constraint qualification condition \cite{Beck_error_bound}:
\BEQ \label{eq:scc} \textstyle \exists\,\bar{\xx}~ \in ~ \bigcap_{i=1}^m \ri(\C_i) ,\EEQ
where $\ri(\C_i)$ denotes the relative interior of $\C_i$. If $\C_i$  is polyhedral then $\ri(\C_i)$ in the above condition can be replaced by $\C_i$.
\ENUM
Note that we have access to oracles for computing Euclidean projections onto each of the following sets: $\X, \C_1, \, \ldots,\,\C_m$ as these sets have been assumed to be simple. Typically, the domain $\X$ is equal to one of the $\C_i$'s. Hence, the availability of projection oracles for $\C_i$'s suffices and no separate oracle is needed for projecting onto $\X$. The standard constraint qualification condition \re{scc} enables us to avoid pathological cases. It is automatically satisfied whenever the feasible set $\C$ has a nonempty interior or all the sets $\C_i$'s are defined by affine equality and inequality constraints. By virtue of assumptions [A1-A3], the set of optimal solutions of \re{p0} is nonempty as $f$ is continuous over the nonempty compact set $\C$. Our goal in this paper is to develop efficient algorithms which can produce for any given $\eps >0$ an $\eps$-optimal $\eps$-feasible solution of \re{p0} with access to only projection oracles of $\X,\C_1,\ldots,\C_m$ and a first-order oracle which returns a subgradient of $f$. Below we provide a brief survey of the existing literature.

\subsection{Related work}

Many algorithms with $O({1}/{\eps^2})$ complexity have been suggested in the stochastic setting \cite{Nedic_Random, multi_constraint_sgd_2015, Bertsekas_incremental_variational}. But these randomized approaches do not provide any insight on how to obtain an approximate solution with a deterministic guarantee on the distance to the feasible set. In \cite{Bertsekas_Incremental} an incremental subgradient approach was proposed for solving general convex optimization problems of the form \re{p0} through an exact penalty reformulation. Though their approach produces an $\eps$-optimal solution of the penalized problem in $O({1}/{\eps^2})$ iterations, such solutions need not be $\eps$-optimal $\eps$-feasible solutions of the original problem \re{p0} as they come with no guarantee on their distance to the feasible set. 

Another line of research considers problem \re{p0} with $\C_i$'s given by functional constraints: $ \C_i = \{ \xx \in \RR^n \,|\, g_i(\xx) \le 0\}$ for some convex function $g_i$. In such setting, the convergence analysis of existing algorithms \cite{Nedic_subgrad_ssp, Mahadavi_one_projection, heavily_constrained_SGD, one_projection_icml17} crucially depends on the assumption $\exists \bar{\xx} \in \RR^n$ such that $g_i(\bar{\xx}) < 0$. Hence, these methods can not be applied for abstract set constraints by taking the distance function $d_{\C_i}$ as $g_i$. Also their dependence on the existence of a strictly feasible point makes them inapplicable in the presence of affine equality constraints. Another short-coming of their approach is its sub-optimal performance when the objective function has smoothness structure. Hence, in this paper we explore alternative approaches without assuming any functional representation for the constraint sets. Using the standard constraint qualification \re{scc} for problem \re{p0} we derive in Section \ref{sec:smp} a primal-dual formulation and an improved convergence guarantee of $O( {1}/{\eps})$ under additional smoothness / structural assumptions on~$f$. In \cite{Poximal_dist} a smooth penalty based approach was proposed for minimizing convex function over intersections of convex sets; however, their approach does not provide any guarantee on the feasibility violation of the approximate solutions. In addition, their method requires the penalty constant to approach infinity, which our method does not require.

A special case of \re{p0} where $\C$ is given by the inverse image of a convex cone under affine transformation was studied in \cite{Lan2013}. Their penalty function based approach does not generalize to other settings. \cite{chiru2015nips} proposed an inexact proximal method to solve a graph transduction problem which is cast as an instance of \re{p0} with $m=2$. They substituted the projection step in the standard subgradient method with an approximate projection which is computed through an iterative algorithm. Due to the use of repeated projections onto $\C_i$'s to compute one approximate projection onto $\C$ their method can be shown to require $O({1}/{\eps^3})$ projections onto each of the $\C_i$s for producing an $\eps$-optimal $\eps$-feasible solution of \re{p0}. 

We mention that when the objective $f$ is strongly convex the fast dual proximal gradient (FDPG) method of \cite{Beck_dual_proximal} can be applied to \re{p0}; \cite{Beck_dual_proximal} showed that the primal iterates (and corresponding primal objective function values) generated by the FDPG method converge to the optimal solution (optimal primal objective value) at $O( {1}/{T})$-rate, where $T$ is the number iterations. But  every iteration of the FDPG method requires solving a subproblem for computing the gradient of the Fenchel-conjugate of $f$. This makes the FDPG method unsuitable for a general strongly convex objective $f$ where $f$ is accessed only through a first-order oracle.

We note that our problem \re{p0} can be posed in the following form for applying splitting methods such as  the alternating direction method of multipliers (ADMM) \cite{Boyd_Linear_Rate_ADMM} or proximal method of multipliers \cite{Teboulle_2014_PMM}:
\BEQ
\label{eq:admm}
\min_{\xx \in \X,\,\bZ}~~ f(\xx) +  \sum_{i=1}^m \ind_{\C_i}(\zz_i)  ~~ \mbox{s.t.} ~~\bA \xx \,= \,\bZ\, ,
\EEQ
where $\bZ \dfn (\zz_1,\ldots,\zz_m)$, $\bA$ denotes the mapping $\xx \mapsto (\xx,\ldots,\xx) \in \otimes_{i=1}^m \RR^{n}$. Note that splitting based approach requires solving a subproblem of the following form at every iteration: $ \argmin_{\xx \in \X} f(\xx)+\frac{\rho}{2}\|\bA\xx -\bZ\|^2$ for a fixed $\rho > 0$ and $\bZ$. Hence, these methods are suitable only when solving the above mentioned subproblem is easy. Therefore, in this paper we aim at developing algorithms which deals with the general case and make no assumption on availability of efficient oracles for solving such subproblems. We mention here that \re{admm} is a special case of semi-separable problem considered in \cite{Nemirovski_mirror_prox_composite}. For that they proposed a first-order algorithm with $O(\frac{1}{\eps}\log\frac{1}{\eps})$ complexity when $f$ possesses special saddle-point structure. Specifically, their algorithm proceeds in stages with each stage solving a saddle-point formulation through composite mirror-prox \cite{Nemirovski_mirror_prox_composite} technique. Our exact penalty based approach enables us to achieve improved complexity of $O(\frac{1}{\eps})$ through a similar mirror-prox based algorithm; notably our approach does not need several stages unlike that of \cite{Nemirovski_mirror_prox_composite}. Additionally, we do not assume the sets $\C_1,\ldots,\C_m$ to be bounded which is needed to apply the method of \cite{Nemirovski_mirror_prox_composite}. 
    
Finally, we mention the connection to the literature on error bounds \cite{Pang_Error_Bound}. For convex feasibility problem (the case when $f = 0$)  there is a rich history of using the distance to the individual sets $\C_1,\ldots,\C_m$ as a proxy for minimizing the distance to the intersection~\cite{Beck_error_bound}. However, we explore the use of the same in the context of optimization problems ($f \neq 0$). Notably, utilizing distance to the individual sets we construct an exact penalty based formulation whose approximate solutions have guarantees on their distance to the intersection of the sets. Note that error bound  properties (characterizations of the distance to the set of optimal solutions) in constrained convex optimization have been shown to hold only for a limited set of problems \cite{Error_Bound_2017}. Assuming certain error bound conditions there have been attempt to establish better convergence rate guarantees \cite{one_projection_icml17}. However, in this paper we deal with the general case with out assuming any error bound property for problem \re{p0}.

\section{Exact Penalty-based Reformulation}
In this section we show that standard constraint qualification \re{scc} allows us to find a suitable penalty function-based reformulation of \re{p0}. Towards that we first recall the concept of \emph{linear regularity} of a collection of convex sets:  
\begin{definition}
The collection of closed convex sets $\{\C_1, \, \ldots,\,\C_m\}$ is linearly regular if $\exists  \Upsilon > 0$ such that  
\BEQ\label{eq:reg_C} \forall \xx \in \RR^n:  ~d_\C(\xx) ~\le~ \Upsilon \,\max_{1 \le i \le m} d_{\C_i}(\xx). \EEQ
\end{definition}
A sufficient condition for $\{\C_1, \ldots,\C_m\}$ to be linearly regular is that $\C = \bigcap_{i=1}^m \C_i$ is bounded and the standard constraint qualification \re{scc} holds \cite{Bauschke1999}. Thus, for problem \re{p0} we have linear regularity of $\{\C_1, \, \ldots,\,\C_m\}$ as a consequence of the assumptions [A1-A3]. In this context we mention that \cite{Nedic_Random, multi_constraint_sgd_2015, Bertsekas_incremental_variational} assumed linear regularity property of the sets for designing stochastic algorithms for problem \re{p0}. For the rest of the paper $\Upsilon$ will denote the linear regularity constant of $\{\C_1, \, \ldots,\,\C_m\}$. Let $R,r>0$ be such that $\C = \cap_{i=1}^m\C_i$ contains a ball of radius $r$ and $\C$ is contained in a ball of radius $R$; then the ratio $R/r$ can be taken as the regularity constant $\Upsilon$ \cite{Regulairty_Infinite_sets}. Please refer to \cite{Bauschke1996} for details about linear regularity and how to estimate the corresponding constant.  In the Appendix we discuss an algorithmic strategy based on a ``doubling trick'' to deal with the case when the regularity constant $\Upsilon$ is not available.

We now discuss the availability of suitable penalty functions for $\C = \bigcap_{i=1}^m \C_i$ such that we can solve the penalty-based reformulation efficiently using existing first-order methods without requiring projection onto~$\C$; however, the method can make use of the oracles for projecting onto each of the $\C_i$s. Below we characterize a class of such penalty functions through the notion of absolute norm \cite{Abs_Norm}.

\begin{definition}
A norm $P$ on $\RR^m$ is called an absolute norm if $\forall \uu \in \RR^m$ we have $P(\uu) = P(|\uu|)$, where $|\uu|$ denotes the vector obtained by taking element-wise modulus of $\uu$. 	
\end{definition}

\begin{proposition}\label{prop:penalty_function} 
Let $P$ be an absolute norm on $\RR^m$ and $h_{\!P} : \RR^n \to \RR_+$ be defined as
\BEQ \label{eq:penalty}
h_{\!P}(\xx) ~=~P(d_{\C_1}\!(\xx),\ldots,\,d_{\C_m}\!(\xx)), ~\xx \in \RR^n.
\EEQ
Then ~$(a)$\, $h_{\!P}$ is a convex function, ~$(b)$\, $h_{\!P}(\xx) = 0$ if and only if $\xx \in \C$, ~$(c)$\, $\exists$ a regularity constant $\Upsilon_{\!P} > 0$ such that
\BEQ \label{eq:reg_h}
\forall \,\xx \in \RR^n:  ~~~ d_{\C}(\xx)  ~\le ~ \Upsilon_{\!\!P} \,h_{\!P}(\xx).
\EEQ
\end{proposition}

With $h_{\!P}$ as defined in \re{penalty}, we consider the following penalized version of problem \re{p0}:
\begin{empheq}[box=\fbox]{align}\label{eq:p1} 
f^{\lambda}_* = \inf_{\xx \in \X}\left[f^{\lambda}(\xx) \,\equiv\, f(\xx) + \lambda\, h_{\!P}(\xx)\right], ~ \lambda > 0.
\end{empheq}
An exact penalty function of the form $\sum_{i=1}^m\gamma_i d_{\C_i}(\cdot)$, where the constants $\gamma_i$ are chosen solely based on the Lipschitz constant $\lip_f$ (without taking linear regularity of the sets into account) was proposed in \cite{Bertsekas_Incremental}. In Appendix we provide a counter example to show that choosing $\gamma_i$'s as per their prescription does not always work (in fact our example shows that Proposition~11 in \cite{Bertsekas_Incremental} does not hold in absence of the standard constraint qualification which we assume in our case). Secondly, the penalty-based reformulation suggested in \cite{Bertsekas_Incremental} does not provide any guarantee on the feasibility violation (distance from the feasible set) of the approximate solutions. In this paper, making use of the standard constraint qualification condition \re{scc} we propose the penalty-based reformulation \re{p1} which we will show to be an exact reformulation of the original problem \re{p0} with the added property that the approximate solutions of \re{p1} are also the approximate solutions of \re{p0} with the desired in-feasibility guarantee. We now show that the use of linear regularity property \re{reg_h} allows us to relate the solution set of \re{p1} to that of \re{p0} and the constant $\lambda$ can be set independent of the desired accuracy of the solution (but big enough) unlike other penalty methods \cite{Poximal_dist} where $\lambda \rightarrow \infty$ is needed.

\begin{proposition}\label{prop:exact_penalty}
Consider problem \re{p0} and the corresponding penalty-based formulation \re{p1} with penalty function $h_{\!P}$ as in \re{penalty}. Then we have:
\BNUM
\item[a.] If $\lambda \ge \Upsilon_{\!\!P} \lip_f$ then $f^{\lambda}_* = f_*$ and every optimal solution of \re{p0} is an optimal solution of \re{p1}.
\item[b.] If $\lambda >  \Upsilon_{\!\!P} \lip_f$ then every optimal solution of \re{p1} is an optimal solution of \re{p0}.
\item[c.] Let  $\lambda \ge 2 \Upsilon_{\!p} \lip_f$ and $\xx_\eps$ be an $\eps$-optimal solution of \re{p1} for a given $\eps > 0$. Then $\xx_\eps$ is an $\eps$-optimal $\eps$-feasible solution of \re{p0}, that is,  $ f(\xx_\eps) - f_*  \le \eps$ and  $d_\C(\xx_\eps) \le \frac{\eps}{\lip_f}.$
Moreover, $\pr_{\C}(\xx_\eps)$ is an $\eps$-optimal feasible solution of \re{p0}.
\ENUM
\end{proposition}

\section{Nonsmooth Objective Functions}

In this section we propose an adaptation of the standard subgradient method for efficiently solving nonsmooth convex optimization problems over intersections of simple convex sets. Specifically, we consider problem \re{p0} with $f$ represented by a black-box oracle of first-order, that is, the oracle returns a subgradient $f'(\xx)$ of $f$ at $\xx \in \X$. Without loss of generality we assume that the subgradients returned by the oracle are bounded by the Lipschitz constant $\lip_f$. Note that direct application of subgradient method to solve \re{p0} requires projection onto the feasible set~$\C$ which may be hard to compute even when $\C$ is given by the intersection of finitely many simple sets. Our adaptation of the  subgradient method, which we call the ``split-projection subgradient'' (\SPS) algorithm, overcomes this difficulty by requiring projections only onto each~$\C_i$. We achieve this by applying the standard subgradient algorithm to problem \re{p1} instead of \re{p0}. In order to apply subgradient method to \re{p1} we present the following Lemma: 
\begin{lemma}\label{lem:subgrad_h}
Let $h_{\!P}$ be as defined in \re{penalty}. Then $h_{\!P}$ is Lipschitz-continuous on $\RR^n$ with Lipschitz constant $P(\one)$ where $\one \in \RR^m$ is the vector of all ones. Moreover, a subgradient of $h_{\!P}$ at $\xx\in \RR^n$ is given by
\BEQ h_{\!P}'(\xx) = \sum_{i=1}^m\frac{u_i^*}{d_i}[\xx - \pr_{\C_i}(\xx)], \nonumber \EEQ
where $\uu^* := \argmax_{\uu \in \RR_+^m}\{ \sum_{i=1}^m u_i d_i \,|\, P_*(\uu)\le 1\}$, $d_i = \|\xx - \pr_{\C_i}(\xx)\|$ and $P_*$ denotes the dual norm of~$P$.
\end{lemma}
The above lemma shows that we can compute a subgradient of the penalty term $h_{\!P}$ utilizing the projection oracles of $\C_1,\ldots,\C_m$. Moreover, the Lipschtiz continuity of~$h_{\!P}$ ensures that such subgradients are bounded by the constant $P(\one)$. Also, recall from the previous section that solving \re{p1} is equivalent to \re{p0} under $\lambda > \Upsilon_{\!\!P} \lip_f$. Therefore, we can apply subgradient method to \re{p1} with $\lambda \ge 2 \Upsilon_{\!\!P} \lip_f$. This results in the \SPS~algorithm for solving \re{p0}. The key recursion in \SPS~algorithm is the following:
\BEQ
\xx^{(t+1)} := \pr_{\X}\!\left( \xx^{(t)} - \gamma_t [f'( \xx^{(t)}) + \lambda \, h_{\!P}'(\xx^{(t)})]\right), \,t\ge1.\nonumber
\EEQ
Algorithmic details are given in the supplementary material. Now, the following proposition states the convergence behavior of the proposed \SPS~algorithm:
\begin{proposition}\label{prop:nonsmooth_convex}
Consider the \SPS~algorithm applied to problem \re{p0} with $\lambda \ge 2 \Upsilon_{\!\!P} \lip_f$. Then, for a given $\eps>0$, \SPS~algorithm produces an $\eps$-optimal $\eps$-feasible solution of \re{p0} in no more than $O ( {1}/{\eps^2} )$ iterations where each iteration involves computation of a subgradient of $f$ and  projections onto each of $\X,\C_1,\ldots,\C_m$.
\end{proposition}
Now, we present an improved complexity estimate for the class of strongly convex functions.

\begin{proposition}\label{prop:nonsmooth_strong}
Consider the \SPS~algorithm applied to problem \re{p0} where $f$ is strongly convex with strong convexity parameter $\mu_{\!f}>0$. Let $\lambda \ge 2 \Upsilon_{\!\!P} \lip_f$. Then, for a given $\eps>0$, \SPS~algorithm produces an $\eps$-optimal $\eps$-feasible solution of \re{p0} in no more than $O( {1}/{\eps})$ iterations where each iteration involves computation of a subgradient of $f$ and  projections onto each of $\X,\C_1,\ldots,\C_m$.
\end{proposition}

\section{Smooth Objective Functions}

In this section we consider solving problem \re{p0} under the additional assumption that $f$ is smooth. Specifically, we assume through out this section that the gradient of $f$, denoted as $\grad f$, is Lipschitz-continuous on $\X$ with Lipschitz constant $M_{\!f}$. Recall that we have access to only a first-order oracle which returns the gradient $\grad f(\xx)$ of $f$ at $\xx \in \X$ and projection oracles for computing projections onto the simple sets $\X,\C_1,\ldots,\C_m$. If we had access to projection oracle for $\C$ then applying accelerated gradient methods we can obtain an $\eps$-optimal solution of \re{p0} in $O(1/\sqrt{\eps})$ iterations. But, in the absence of projection oracle for~$\C$, problem \re{p0} is essentially an instance of nonsmooth optimization as the nonsmooth part $\sum_{i=1}^m\ind_{\C_i}$ does not possess a tractable proximal operator. Therefore, existing first-order methods for smooth/composite convex minimization can not be applied directly to problem \re{p0}. 

One of the main contributions of the paper is  to show that we can use first-order methods
through an adaptation of the primal-dual framework of \cite{Pock_2016_ergodic}. To apply the primal-dual framework we first propose a saddle-point reformulation of \re{p1} by exploiting the structure of the nonsmooth penalty function $h_{\!P}$. Before going into the details, we introduce the following notation: $\bY \dfn (\yy_1,\ldots,\yy_m) \in \otimes_{i=1}^m \RR^n$. We have:
\begin{lemma}\label{lem:hp_dual}
Let $h_{\!P}$ be as defined in \re{penalty}. Then the following holds for all $\xx \in \RR^n$:
\BEQ \label{eq:hp_dual} h_{\!P}(\xx) = \max_{\bY \in \Y_P} \sum_{i=1}^m \left[ \xx\t\yy_i - \supp_{\C_i}(\yy_i) \right], \EEQ
where $\Y_P \dfn \{ \bY \in  \otimes_{i=1}^m \RR^n \,|\,P_*(\|\yy_1\|,\ldots,\|\yy_m\|) \le 1\}$, $P_*$ denotes the dual norm of $P$ and $\supp_{\C_i}$ denotes the support function of set $\C_i$. 
\end{lemma}
Exploiting the above structure of $h_{\!P}$ we have the following saddle-point reformulation of \re{p1} for any $\lambda>0$:
\BEQ \label{eq:p2} 
\min_{\xx \in \X} \max_{\bY \in \Y^{\lambda}_P}\Big[\L(\xx,\bY) \equiv f(\xx)+\sum_{i=1}^m\xx\t\yy_i - g(\bY)\,\Big],
\EEQ
where $\Y^{\lambda}_P \dfn \{ \bY \in  \otimes_{i=1}^m \RR^n \,|\,P_*(\|\yy_1\|,\ldots,\|\yy_m\|) \le \lambda\},$
\BEQ \label{eq:gY} g(\bY) ~\dfn~ \sum_{i=1}^m\supp_{\C_i}(\yy_i) + \ind_{\Y^{\lambda}_P}(\bY), ~~\bY\in \otimes_{i=1}^m \RR^n.\EEQ
We can now connect the saddle-point formulation \re{p2} with the original problem~\re{p0}. 
\begin{lemma}\label{lem:sad_sol}
Consider the saddle-point formulation \re{p2} with $\lambda \ge 2 \Upsilon_{\!\!P}\lip_f$. Fix $\eps>0$. Let $(\xx_\eps,\bY_{\!\eps}) \in \X \times \Y^{\lambda}_P$ be an $\eps$-optimal solution of \re{p2} in the following sense:
\BEQ \label{eq:sad_eps}
\sup_{\xx \in \X,\,\bY\in\Y^{\lambda}_P} ~[~\L(\xx_\eps,\bY ) - \L(\xx,\bY_{\!\eps})~] ~\le~ \eps.
\EEQ
Then $\xx_\eps$ is an $\eps$-optimal $\eps$-feasible solution of \re{p0}.
\end{lemma}
In order to solve \re{p2} efficiently, the following lemma  shows that the proximal operator of the nonsmooth convex function $g$ can be evaluated in closed form through projections onto the sets $\C_1,\ldots,\C_m$.
\begin{lemma}\label{lem:prox_g}
Let $g$ be  defined in \re{gY}. Then for any $\gamma > 0$ and $\bY \in \otimes_{i=1}^m \RR^n$ the proximal operator of $g$ is given by
$\prox_{\gamma g}(\bY) = (r_1\hat{\yy}_1/\|\hat{\yy}_1\|,\ldots,r_m\hat{\yy}_m/\|\hat{\yy}_m\| )$, where $\hat{\yy}_i \dfn \yy_i - \gamma \pr_{\C_i}(\gamma^{-1}\yy_i)$ and $(r_1,\ldots,r_m)$ is the projection of $(\|\hat{\yy}_1\|,\ldots,\|\hat{\yy}_m\|)$ onto $\{ \uu \in \RR^m \,|\, P_*(\uu) \le \lambda \}$.
\end{lemma}
With the proximal operator of $g$ being computable and $f$ being smooth, we define a primal-dual iteration of the following form: 
\BEA
&\mbox{Iteration:}~(\xx^{+},\bY^{+}) = \mathcal{PD}_{\tau,\gamma}(\xx,\bY,\tilde{\xx},\tilde{\bY}) \nonumber\\
&\left\{
\begin{array}{l}
\xx^{+} := \pr_{\X}\!\left( \xx- \tau \left[\grad f(\xx) + \sum_{i=1}^m \tilde{\yy}_i\right]\right), \\ 
\bY^{+} := \prox_{\gamma g}\!\left(\bY+\gamma \bA\tilde{\xx}\right), 
\end{array}\label{eq:pd_iter}
\right.
\EEA
where $\bA$ denotes the map $\xx \mapsto (\xx , \ldots,\xx) \in \otimes_{i=1}^m\RR^n$. With this we can now apply primal-dual algorithms of \cite{Pock_2016_ergodic} to problem \re{p2} with primal-dual iteration defined by \re{pd_iter}. Since $\X$ and $\Y^{\lambda}_P$ are compact sets, we can obtain an $\eps$-optimal solution of \re{p2} in the sense of \re{sad_eps} by applying $O(1/\eps)$ iterations of the non-linear primal-dual algorithm of \cite{Pock_2016_ergodic}. Moreover, when $f$ is smooth as well as strongly convex we can apply the accelerated primal-dual algorithm of \cite{Pock_2016_ergodic} which needs only $O(1/\sqrt{\eps})$ iterations of the form \re{pd_iter}. Note that Lemma~\ref{lem:sad_sol} guarantees that such solutions are enough to output an $\eps$-optimal $\eps$-feasible solution of \re{p0}. Therefore, complexity of obtaining an $\eps$-optimal $\eps$-feasible solution of \re{p0} is $O(1/\sqrt{\eps})$ when $f$ is smooth and $O(1/\sqrt{\eps})$ for smooth and strongly convex $f$. Thus, utilizing existing primal-dual machinery to a saddle-point reformulation of the exact penalty based equivalent problem \re{p1}, we achieve better complexity for problems of the form \re{p0} under smoothness assumption on~$f$. We call this approach exact penalty primal-dual (EPPD) method.

\section{Nonsmooth Objective Functions with Structure} \label{sec:smp}
In many machine learning problems such as kernel learning \cite{Bach_MKL}, learning optimal embedding for graph transduction~\cite{chiru2015nips}, etc., the objective function $f$ is defined as the optimal value of a maximization problem. In most of these cases, in spite of $f$ being non-smooth, the problem of minimizing $f$ can be cast as a smooth saddle-point problem. It is well-known that by exploiting such structure in the problem, first-order algorithms with improved convergence rate of $O(1/\eps)$ can be obtained even for non-smooth problems \cite{Nemirovski_MLopt2}. Hence, in the context of problem \re{p0} we would like to address the following question: by exploiting structure in $f$ is it possible to design first-order algorithms with $O(1/\eps)$ complexity for \re{p0}? For this we make the following additional assumptions on problem \re{p0}. Through out this section we will assume that the objective function $f$ possesses the following structure:
\BEQ f(\xx) ~= ~  \max_{ \zz ~\in ~\Z} ~F(\xx,\zz), ~\xx \in \X, \label{eq:spf} \EEQ
%\BEQ f(\xx) ~= ~  \max_{ \zz ~\in ~\Z} \,\Big[~ \phi(\xx,\zz) +\psi_1(\xx)- \psi_2(\zz) ~\Big], ~\xx \in \X, \label{eq:spf} \EEQ
where $\Z$ is a simple compact convex set and $F: \X \times \Z \to \RR$ is a convex-concave function %(convex in $\xx$ and concave in $\zz$)
with Lipschitz continuous gradient. Also, we assume availability of a first-order oracle for computing the gradient of $F$ at any $(\xx,\zz)\in \X \times \Z$. 
% \BIT
%\item $\Z$ is a simple compact convex set in some Euclidean space $\RR^K$,
%\item $\psi_1: \X \to \RR, ~\psi_2: \Z \to \RR$ are closed convex (possibly non-smooth) functions with efficient oracles for computing their proximity operators, 
%\item $F: \X \times \Z \to \RR$ ~is convex-concave with Lipschitz continuous gradient on $\X \times \Z$, 
%item a first-order oracle for computing the gradient of $F$ at any point in $\X \times \Z$ is available.  
%\EIT

Recall that only projections onto $\X,\C_1,\ldots,\C_m$ are available and algorithms can not ask for projections onto~$\C$. This makes the existing mirror-prox algorithm \cite{Nemirovski_MLopt2} unsuitable for problem \re{p0} even when $f$ has the above structure. We overcome this difficulty by considering the penalty based formulation \re{p1} where $\lambda \ge 2 \Upsilon_{\!\!P}\lip_f$, $f$ as in \re{spf} and $h_P$ given by \re{hp_dual}. This results in the following saddle-point formulation like~\re{p2}:
\BEQ \label{eq:p3} 
\min_{\xx \in \X}\, \max_{\zz\in\Z,\,\bY \in \Y^{\lambda}_P}~ ~[\,F(\xx,\zz)+\sum_{i=1}^m\xx\t{\yy_i} - g(\bY)\,].
\EEQ
We see that the objective above has a nonsmooth term $g$ and the remaining part has Lipschitz continuous gradient. Also, as given in Lemma~\ref{lem:prox_g}, the proximal operator of $g$ can be computed through projections onto $\C_i$'s. Hence, the mirror-prox-``a'' (MPa) algorithm \cite{Nemirovski_MLopt2} can be applied to problem \re{p3}. The resulting approach we call split-mirror prox (\MPICS), with the following   complexity estimate:
\begin{proposition}\label{prop:MPICS}
Given $\eps > 0$, \MPICS~algorithm requires no more than $O( {1}/{\eps})$ calls to the first order oracle of $F$ and $O( {1}/{\eps})$ projections onto each of $\X,\Z,\C_1,\ldots,\C_m$ to obtain an $\eps$-optimal $\eps$-feasible solution of \re{p0} where $f$ is of the form \re{spf}.
\end{proposition}

\def\svm{{\omega_{C}(\bK,\by_S)}}
\def\cK{\mathcal{K}}
\def\cS{\mathcal{ S}}
\def\by{\textbf y}
\section{Experimental Results}
\label{sec:expt}
%%%%%%%%%%%%%%%%%%%%%%%%%%%%%%%%%%
In this section we illustrate the benefits of the proposed algorithms on two problems: a graph transduction problem where the objective function is non-smooth but can be cast in the form \re{spf} and on a graph matching problem where the objective is a smooth function with Lipschitz continuous gradient. We performed all the experiments on a CPU with Intel Core i7 processor and 8GB memory. In the implementation of the proposed methods we choose the norm $P$ to be the standard $\ell_1$-norm.

\subsection{Learning orthonormal embedding for graph-transduction}
Consider a simple graph $G=(V,E)$, with vertex set $V=\{1,\ldots,N\}$ and edge set $E \subset V \times V$. If $S \subseteq V$ is labelled with binary values denoted by $\by_S \in \{-1,1\}^{|S|}$~the problem of graph transduction can be posed as learning the labels of the remaining vertices. Recently the following problem was posed in \cite{chiru2015nips} for learning the optimal orthonormal embedding of the graph for solving the problem of graph transduction with very encouraging results.
 \BEQ \label{eq:spore} \min_{\bK \in \cK(G)} \svm \,+\, \beta\, \lambda_{max}(\bK), \EEQ 
 where 
\BEA
& \svm=\max_{\alphab\in \A} \sum\limits_{i\in S}\alpha_i - \frac{1}{2}\sum\limits_{i,j\in S}\alpha_i\alpha_jy_iy_jK_{ij},\nonumber \\
&\A = \{ \alphab \in \RR^N \,|\, 0 \le \alpha_i \le C~\forall i\in S,\; \alpha_j =0~\forall j\notin S \}.\nonumber
\EEA
$\bK$~is a positive semidefinite (PSD) kernel matrix arising due to an orthonormal embedding characterized by the following set
$${\bf \mathcal{K}(G)}:=\big\{\bK\in \SS^N_+ \,|\, K_{ii}=1\, \forall i, K_{ij}=0\, \forall (i,j)\notin E\big\},$$
where $\SS^N_+$ denotes the class of symmetric PSD matrices in $\RR^{N\times N}$. The set $\cK(G)$ ~is an elliptope lying in the intersection of PSD cone with affine constraints. The objective function consists of two nonsmooth functions, $\svm$ and ~$\lambda_{max}(\bK)$, the largest eigenvalue of $\bK$ where $\beta > 0$ is user defined. In \cite{chiru2015nips} an inexact infeasible proximal method (IIPM) was proposed which do not provide any feasibility guarantee on the approximate solutions. To illustrate the effect of feasibility we compare the proposed \MPICS~method with IIPM on solving \re{spore}. 
\begin{table}
\centering
\caption{Comparison of classification accuracy (mean $\pm$ standard deviation) on MNIST digit recognition dataset.}
\setlength{\tabcolsep}{10pt}
\vspace{5pt}
\begin{tabular}{|c||c|c|}
\hline
{Dataset} & { \MPICS } & { IIPM }  \\   \hline \hline
1 vs 2  & {\bf 96.9} $\pm$ 0.4 &   96.2 $\pm$ 1.2 \\ 		
1 vs 7  & {\bf 97.6} $\pm$ 0.5 &  95.0  $\pm$ 3.0 \\ 
3 vs 8  & {\bf 89.7} $\pm$ 1.3 &  86.8 $\pm$ 2.7 \\ 
4 vs 9  & {\bf 83.0} $\pm$ 1.9 &  77.5 $\pm$ 2.1 \\ 	
6 vs 8  & {\bf 97.6} $\pm$ 0.3 &  93.8 $\pm$ 2.0 \\ \hline
\end{tabular}  \label{table:accuracy_table}
\end{table}

\begin{table}
\centering
\caption{Comparison of objective function value on MNIST digit recognition dataset. }
\setlength{\tabcolsep}{10pt}
\vspace{10pt}
\begin{tabular}{|c||c|c||c|c|} 
\hline
\multicolumn{1}{|c||}{ Dataset} &
\multicolumn{2}{c||}{  Infeasible	
 } &
\multicolumn{2}{c|}{  Feasible }   \\   \cline{2-5}
\multicolumn{1}{|c||}{  } &
\multicolumn{1}{c|}{ \MPICS } &
\multicolumn{1}{c||}{ IIPM } &
\multicolumn{1}{c|}{ \MPICS } &
\multicolumn{1}{c|}{ IIPM }   \\   \hline \hline
1 vs 2  &  5.77  & 3.97  &  {\bf 5.77}  &  5.80  \\ 
1 vs 7  &  5.31  & 3.30  &  {\bf 5.32}  &  5.33  \\ 
3 vs 8  &  6.46  & 4.33  &  {\bf 6.46}  &  6.54  \\ 
4 vs 9  &  6.13  & 4.19  &  {\bf 6.14}  &  6.18  \\ 
6 vs 8  &  6.35  & 4.38  &  {\bf 6.35}  &  6.38  \\  
 \hline
\end{tabular}  \label{table:obj_val_table}
\end{table}

\begin{figure}[h]   \centering
\includegraphics[scale =.5]{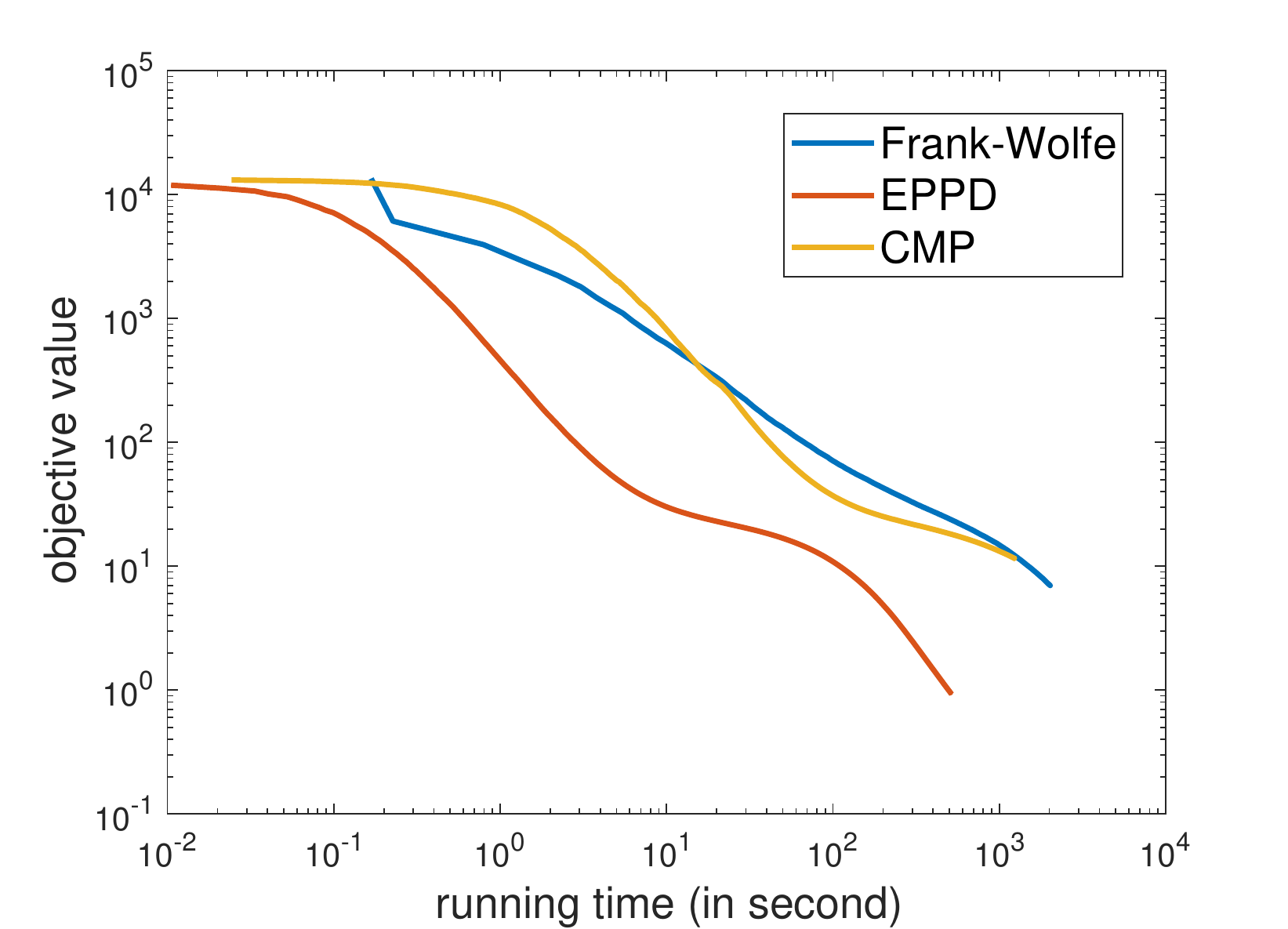} 
%\vspace*{-.45cm}
\caption{Convergence plot of Frank-Wolfe, Exact Penalty Primal-Dual (EPPD) and Composite Mirror-Prox (CMP) method on the Graph Matching problem. \label{fig:convergence}
}
\end{figure}

We experimented on a subset of the MNIST dataset~\cite{MNIST} where corresponding to each pair of digit classes we constructed a graph with $n=1000$ nodes as follows: (a) first randomly select 500 samples from each digit; (b) for each pair of samples put an edge in the graph if the cosine distance between samples is less than a threshold value (we set 0.4 as the threshold). For IIPM the number of inner-iteration (S) to compute approximate projection was set to 5. The regularization parameters $\beta$ and $C$ were selected through 5-fold cross-validation. Table~\ref{table:accuracy_table} summarizes the results, averaged over 5 random training/test partitioning. Labels for 10 percent of the nodes were used for training. Entries in the table represent classification accuracy (mean $\pm$ standard deviation) which we calculate as the percentage of un-labelled nodes classified correctly. To compare the effect of in-feasibility of the iterates generated by the two methods we report in Table~\ref{table:obj_val_table} the objective function value reached for one particular training/test partitioning of the data. Under the infeasible column we present the objective value at the infeasible solutions returned by \MPICS / IIPM whereas in the feasible case we project the output from both the algorithms onto the feasible set before computing the objective value. We observe that IIPM misleadingly reports smaller objective value; this is result of the iterates being far from the feasible set. As the new \MPICS~algorithm ensures that iterates are not far from the feasible set; also, objective function values do not change much even after projecting the infeasible solution onto the feasible set. Also, we see better predictive performance for \MPICS~in Table~\ref{table:accuracy_table}.

\subsection{Graph matching}
We consider a graph matching problem, where two adjacency matrices $A$ and $B$ in $\RR^{n\times n}$ are given, and we aim to minimize $\| A \Pi - \Pi B\|_F^2$ with respect to $\Pi$ over the set of doubly stochastic matrices. This can be seen as a natural convex relaxation of optimizing on the set of permutation matrices~\cite{Graph_Matching_Bach}. The set of doubly stochastic matrices is defined as the intersection of two products of $n$ simplices in dimension $n$ (which are indeed simple sets with efficient projection oracles). We compare the proposed Exact Penalty based Primal-Dual (\EPPD) approach to the Frank-Wolfe algorithm, for which the linear maximization oracle is an assignment problem, which can be solved in $O(n^3)$. Note that both algorithms have the same convergence rate in terms of number $t$ of iterations, as $O(1/t)$. We also include in comparison the Composite Mirror-Prox (CMP) based approach for solving semi-separable problems \cite{Nemirovski_mirror_prox_composite}. We present experimental results on randomly generated undirected graphs with number nodes = 200. In Figure~\ref{fig:convergence} we compare the convergence behavior of the methods. Although all the algorithms have very similar convergence rate, our \EPPD~method takes considerably lesser time. This is due to the fact the \EPPD~method just need projection onto simplices where as Frank-Wolfe needs to solve a linear maximization problem over the set of doubly stochastic matrices which is computationally more demanding. As the composite Mirror-Prox method requires two gradient computations and 2 proximal evaluations per iteration, every iteration of CMP is at least twice as costly as our primal-dual iteration; moreover, their formulation introduces $m$ more variables; this makes CMP approach slower than EPPD.

\section{Conclusions}
In this paper, we presented algorithms to minimize convex functions over intersections of simple convex sets, with explicit convergence guarantees for feasibilty and optimality of function values. Our work not only bounds the level of in-feasibility, currently missing in existing literature but also improves the convergence rate. This is mostly based on a new saddle-point formulation with an explicit proximity operator, and led to improved experimental behavior in two situations. Our work opens up several avenues for future work: (a) we can imagine letting the number $m$ of sets grow large or even to infinity and using a stochastic oracle~\cite{Nedic_Random} with efficient stochastic gradient techniques~\cite{sag}, (b) we could consider other geometries than the Euclidean one by considering mirror descent extensions. 

% Acknowledgement
\acks{This work was supported by a generous grant under the CEFIPRA project and the INRIA associated team BIGFOKS2. }

\bibliography{projections}
\bibliographystyle{plain}
%\bibliographystyle{unsrtnat}
%\setcitestyle{square,aysep={ },yysep={;}}

\newpage

\section{Appendix}

Before going into the proofs of the propositions and lemmas stated in the paper, we state the following proposition:

\begin{proposition}\label{prop:dist}
	Let $\A$ be a nonempty closed convex set in $\RR^n$. The distance function $d_{\A}$ given by 
	\BEQ  \label{eq:d_A} d_{\A}(\xx) \dfn  \inf_{\aa \in \A}\|\xx-\aa\|, \,\xx \in \RR^n, \EEQ
	has the following properties:
	\BNUM
		\item[1.] $d_{\A}$ is convex and Lipschitz continuous on $\RR^n$ with Lipschitz constant $1$.
		\item[2.] $d_{\A}$ has the following representation:  
		\BEQ d_{\A}(\xx) =  \sup_{\yy\in\RR^n:\|\yy\| \le 1 } [\xx\t\yy-\supp_{\A}(\yy)], ~\xx \in \RR^n, \label{eq:d_dual}\EEQ
		where $\supp_\A$ denotes the support function of $\A$.
		\item[3.] $\frac{\xx - \pr_{\A}(\xx)}{\|\xx - \pr_{\A}(\xx)\|}$ is a subgradient of $d_{\A}$ at $\xx \in \RR^n$.
	\ENUM
\end{proposition}

\begin{proof}
	The convexity of $d_\A$ is clear from \re{d_A} as $(\xx,\aa) \mapsto \|\xx-\aa\|$ is jointly convex in $\xx$ and $\aa$. 
	
	Let $\xx,\xx' \in \RR$ and $\aa \in \A$. By triangle inequality of the Euclidean norm, we have:
	$$\|\xx - \aa\| ~\le~\|\xx-\xx'\|+\|\xx'-\aa\|.$$
	By taking minimum on both sides over $\aa \in \A$ we get
	$$ d_{\A}(\xx) \le \| \xx -\xx' \| + d_{\A}(\xx')$$
	Now, interchange the role of $\xx$ and $\xx'$ to arrive at the following:
	\BEQ \label{eq:d_lip}\forall \xx,\xx'\in \RR^n:~~ | d_{\A}(\xx) - d_{\A}(\xx')| \le \| \xx -\xx' \|. \EEQ
	This shows that $d_{\A}$ is Lipschitz continuous with Lipschitz constraint $1$.
	
	To prove the 2nd part we have for any $\xx \in \RR^n:$
	\BEA 
		 d_{\A}(\xx) &=& \min_{ \aa \in \A} \, \| \xx - \aa \|   \nonumber \\
		 &=& \min_{ \aa \in \A} ~\max_{ \yy\in\RR^n: \| \yy \| \le 1}  ~\yy\t(\xx - \aa) \label{eq:d_norm} \\
		 &=& \max_{ \yy\in\RR^n: \| \yy \| \le 1} ~\min_{ \aa \in \A} ~ \yy\t(\xx - \aa) \nonumber \\
		 &=& \max_{ \yy\in\RR^n: \| \yy \| \le 1}  ~[ \yy\t\xx - \max_{ \aa \in \A} \yy\t \aa ] \nonumber \\
		 &=& \max_{ \yy\in\RR^n: \| \yy \| \le 1}  ~[\xx\t\yy-\supp_{\A}(\yy)], \label{eq:d_final}
	\EEA
	where the 3rd equality follows from Min-Max Theorem \cite{Minmax_Sion} as $\{ \yy\in\RR^n: \| \yy \| \le 1\}$ is compact.
	
	Since, $\pr_{\A}(\xx)$ denotes the Euclidean projection of $\xx$ onto $\A$. We have $d_{\A}(\xx) = \| \xx - \pr_{\A}(\xx)\| = (\xx - \aa^*)\t\yy^*$, where $\aa^* = \pr_{\A}(\xx)$ and $\yy^* = \frac{\xx - \pr_{\A}(\xx)}{\|\xx - \pr_{\A}(\xx)\|}$. Therefore, $(\aa^*,\yy^*)$ is a saddle-point of \re{d_norm}. So, $\yy^*$ is an optimal solution of \re{d_final}. Now, note that any maximizer $\yy$ of \re{d_final} is a subgradient of $d_{\A}$ at $\xx$. This completes the proof of the 3rd part of the proposition. 
\end{proof}

\subsection{Proof of Proposition \ref{prop:penalty_function}}

	Since, $P$ is an absolute norm, it is non-decreasing in the absolute values of its components~\cite{Abs_Norm}. Therefore, (a) follows from the convexity of the distance functions $d_{\C_i}$ (part 1 of Proposition~\ref{prop:dist}) and monotonicity of the norm $P$.
	
	To prove part (b) we note that $\xx \in \C$ iff $d_{\C_i}(\xx) = 0\, \forall i$. Also, $h_{\!P}(\xx)$ is zero iff $d_{\C_i}(\xx) = 0\,\forall i$ as $P$ is a norm. 
	
	Recall that $(\C_1,\ldots,\C_m)$ is linearly regular with constant $\Upsilon$. Therefore, as a consequence of \re{reg_C}, we have \re{reg_h} with $\Upsilon_{\!\!P} = \Upsilon \,\max\{  \| \uu \|_{\infty} : P(\uu) = 1\}$, where $\| \cdot \|_{\infty}$ denotes the standard $\ell_{\infty}$-norm on $\RR^m$.

\subsection{Counter Examples for Proposition 11 of \cite{Bertsekas_Incremental}}

	Here we present counter examples to falsify the claims made by \cite{Bertsekas_Incremental} in their Proposition 11. We first show that their proof of Proposition 11 does not hold always. Then we present another counter example to prove that their Proposition 11 is not true in general, specifically, when standard constraint qualification condition \re{scc} is violated. 
	
	Let $X_1,\,\ldots,\,X_m$ be closed convex subsets of $Y \subset \RR^n$ with nonempty intersection and $f:Y\to \RR$ be Lipschitz continuous with Lipschitz constant $\lip_f$\,. We now state the incorrect claims of \cite{Bertsekas_Incremental} supported by our counter examples.

	{\bf Claim 1:} The construction given in proof of Proposition 11 in \cite{Bertsekas_Incremental} claims that the set of minima of $f$ over $\cap_{i=1}^m X_i$ coincides with the set of minima of 
	\BEQ F(\xx) ~~\dfn~~ f(\xx) \,+\, \gamma \, \sum_{i=1}^m d(\xx,X_i) \label{eq:Fx}\EEQ
	over $Y$ if 
	\BEQ \gamma_0 = 0,~\forall k\ge1:\,\gamma_k \, > \,\lip_f + \sum_{i=1}^{k-1}\gamma_i, ~\mbox{and}~\gamma \ge \gamma_m . \label{eq:gamma} \EEQ
	
	{\bf Counter Example 1:} We present the following counter example to show that the above claim of \cite{Bertsekas_Incremental} is not always true. Consider the following set-up: \\
	$$n=2,~Y = \RR^2, ~m=2,$$
	$$X_1 \,=\, \{ (x,y) \in \RR^2 \,|\, 0.1 x + y \le 1 \},$$ 
	$$X_2 \,=\, \{ (x,y) \in \RR^2 \,|\, 0.1 x - y \le 1\}, $$ 
	$$f(x,y) =  -x-y, ~~\forall (x,y) \in \RR^2.$$
	Clearly, $f$ is Lipschitz continuous with Lipschitz constant $\lip_f =\sqrt{2}$. Now, satisfying the conditions given in \re{gamma}, we choose $\gamma_1 = 1.5$, $\gamma_2 = 3$ and $\gamma = 4$. With this $F$ defined in \re{Fx} becomes:
	\BEQ F(x,y) = -x-y+\frac{4([0.1 x+y-1]_+ + [0.1 x-y-1]_+)}{\sqrt{1.01}}, \nonumber\EEQ
	where $[a]_+ = \max\{a,0\}$ for any $a \in \RR$. \\
	Note that $(10,0)$ is the only minima of $f$ over $X_1 \cap X_2$. But, $(10,0)$ can not be a minima of $F$ as we have $F(20,0) < F(10,0)$. In fact, $F$ does not have any minima on $\RR^2$ as $F(x,0) \to -\infty $ when $x \to \infty$. Hence, the set of minima of $F$ over $Y$ need not be the same as the set of minima of $f$ over $\cap_{i=1}^m X_i$ even if $\gamma$ satisfies the condition given in \re{gamma}. Thus, the proof of {Proposition 11} in \cite{Bertsekas_Incremental} stands void.
	
	We mention that the above example possesses the following linear regularity property: 
	\BEQ \forall \xx \in \RR^2:~ d(\xx,X_1\cap X_2) ~\le~ \Upsilon \, \max_{1 \le i \le 2} d(\xx,X_i)\,, \label{eq:linreg} \EEQ
	where $\Upsilon = 1 / \sin( \tan^{-1}(0.1))$. So, for the above example one can be verify that setting $\gamma > \Upsilon \lip_f$ in \re{Fx} suffices for the set of minima of $F$ over $Y$ to coincide with that of $f$ over $X_1 \cap X_2$. 
	
	{\bf Claim 2:} Proposition 11 in \cite{Bertsekas_Incremental} claims that $\exists \bar{\gamma} > 0$ such that the set of minima of $f$ over $\cap_{i=1}^m X_i$ coincides with the set of minima of $F$ (as defined in \re{Fx}) over $Y$ for all $\gamma \ge \bar{\gamma}$.
	
	{\bf Counter Example 2:} We construct the following counter example to show that the above claim of \cite{Bertsekas_Incremental} is false. Consider the following set-up: \\
	$$n=2,~Y = \RR^2, ~m=2,$$
	$$X_1 \,=\, \{ (x,y) \in \RR^2 \,|\, y \ge x^2 \},$$ 
	$$X_2 \,=\, \{ (x,y) \in \RR^2 \,|\, y=0 \}, $$ 
	$$f(x,y) =  -2 x, ~~\forall (x,y) \in \RR^2.$$
	Clearly, $f$ is Lipschitz continuous with Lipschitz constant $\lip_f =2$. We note that $ X_1 \cap X_2 = \{ (0,0)\}$. Therefore, $(0,0)$ is the only minima of $f$ over $X_1 \cap X_2$. Fix any $\gamma >0$ in \re{Fx}. Now, for all $x \in \RR$ we have
	\BEQ  F(x,x^2) =  -2x + \gamma x^2 .\nonumber \EEQ
	Setting $x = \gamma^{-1}$ we have 
	\BEQ F( \gamma^{-1}, \gamma^{-2}) = - \gamma^{-1} ~<~ F(0,0).\nonumber \EEQ
	Therefore, $(0,0)$ is not a minima of $F$ over $\RR^2$ for any $\gamma > 0$. Hence, $\nexists \bar{\gamma} > 0$ such that $(0,0)$ is a minima of $F$ for any $\gamma \ge \bar{\gamma}$. This establishes that claim 2 is not true in general. 
	
	One can verify that $\nexists \Upsilon >0$ such that \re{linreg} holds where $X_1, X_2$ are as in example 2. We mention that standard constraint qualification (SCQ) condition \re{scc} is not satisfied in this example. Recall that SCQ, although a very mild requirement, is sufficient to ensure linear regularity property when the intersection of the closed convex sets is bounded.

\subsection{Proof of Proposition \ref{prop:exact_penalty}}

	Recall that $f_* = \min_{\xx\in\C}f(\xx)$. As $f$ is Lipschitz continuous on $\X$ with Lipschitz constant $\lip_f$, we have $\forall \rho \ge \lip_f:$
	\BEA 
	&\forall& \,\xx,\yy \in \X:~ f(\yy) - f(\xx) ~\le ~ \rho \, \| \yy - \xx \|,  \nonumber \\
	\Rightarrow \, &\forall& \xx \in \X:~ \inf_{ \yy \in \C}f(\yy) - f(\xx) ~\le ~ \rho \,\inf_{\yy \in \C} \| \yy - \xx \|, \nonumber \\
	\Rightarrow \, &\forall& \xx \in \X:~ f_* ~\le ~ f(\xx) \,+\, \rho \,d_\C(\xx). \label{eq:penalty_lemma} 
	\EEA
	Using regularity of $h_{\!P}$ from \re{reg_h}, we have $\forall \lambda \ge 0:$
	\BEQ\label{eq:hd_ineq} f^{\lambda}_* = \inf_{\xx \in \X} f(\xx) + \lambda \,h_{\!P}(\xx) \ge \inf_{\xx \in \X} f(\xx) + \frac{\lambda}{\Upsilon_{\!\!P}}\, d_\C(\xx).\EEQ
	
	If $\lambda \ge  \Upsilon_{\!\!P} \lip_f$ then applying \re{penalty_lemma} in \re{hd_ineq}, we obtain $f^{\lambda}_* \ge f_*$. On the other hand, we have $f^{\lambda}_* \le f_*$ for all $ \lambda \ge 0$ as $\C \subset \X$ and $h_{\!P}$ is zero on $\C$. Thus, $f^{\lambda}_* = f_*$ and a minima of $f$ over $\C$ is also a minima of $f^{\lambda}$ over $\X$.
	
	To prove part [b] of the proposition it is enough to show the following: if $\xx_\lambda^*$ is an optimal solution of \re{p1} then $\xx_\lambda^* \in \C$ when $\lambda >  \Upsilon_{\!\!P} \lip_f$. Assume $\xx_\lambda^* \notin \C$. So, $d_\C(\xx_\lambda^*) > 0$. Now, using \re{reg_h} and $\lambda >  \Upsilon_{\!\!P} \lip_f$ we have
	\BEA f^{\lambda}_* = f(\xx_\lambda^*) + \lambda\, h_{\!P}(\xx_\lambda^*) &\ge& f(\xx_\lambda^*) + \frac{\lambda}{\Upsilon_{\!p}}\, d_\C(\xx_\lambda^*)  \nonumber \\
	&>& f(\xx_\lambda^*) + \lip_f \, d_\C(\xx_\lambda^*). \nonumber 
	\EEA
	Now, applying \ref{eq:penalty_lemma} we obtain $f^{\lambda}_* > f_*$ which is a contradiction to the first part of the theorem. Thus, every minimizer $\xx_\lambda^*$ of $f^{\lambda}$ over $\X$ must belong to $\C$  when $\lambda >  \Upsilon_{\!\!P} \lip_f$. This together with the fact that $f^\lambda_* =f_*$ shows that $\xx_\lambda^*$ is also a minimizer of $f$ over $\C$. 
	
	Now, we focus on part [c] of the proposition. By definition of $\eps$-optimal solution of \re{p1}, we have
	\BEQ f(\xx_\eps) + \lambda\, h_{\!P}(\xx_\eps) \,\le\, f^{\lambda}_* + \eps. \label{eq:bound0} \EEQ
	Now, applying \re{reg_h} and using $f^{\lambda}_* = f_*$ from part [a], we get
	\BEQ f(\xx_\eps) + \frac{\lambda}{ \Upsilon_{\!\!P} }\,d_\C(\xx_\eps) \,\le\, f_* +\eps. \label{eq:bound1} \EEQ
	Putting $\rho =\lip_f$ in \re{penalty_lemma} we have
	\BEQ f(\xx_\eps) + \lip_f\,d_\C(\xx_\eps) \,\ge\, f_*. \label{eq:bound2} \EEQ
	Now, combining \re{bound1}, \re{bound2} and using $\lambda \ge 2 \Upsilon_{\!\!P} \lip_f$ we achieve $$d_{\C}(\xx_\eps) \le \frac{\eps  \Upsilon_{\!\!P}}{\lambda - \Upsilon_{\!\!P} \lip_f} \le  \frac{\eps}{\lip_f}.$$
	Note that we also have $f(\xx_\eps)-f_* \le \eps$ from \re{bound0} as $ h_{\!P}$ is always non-negative. This proves that $\xx_\eps$ is also an $\eps$-optimal $\eps$-feasible solution of \re{p0}. Now it remains to show that $f(\pr_{\C}(\xx_\eps)) \le f_*+\eps$. This follows from Lipschitz continuity of $f$ and \re{reg_h} as shown below:
	\BEA f(\pr_{\C}(\xx_\eps))  &\le& [f(\pr_{\C}(\xx_\eps)) - f(\xx_\eps)] + f(\xx_\eps) \nonumber \\
	&\le& \lip_f \,\| \pr_{\C}(\xx_\eps) - \xx_\eps \| + f(\xx_\eps) \nonumber \\
	&=& \lip_f  \, d_{\C}(\xx_\eps)+ f(\xx_\eps) \nonumber \\
	&\le& \lip_f \Upsilon_{\!\!P}\,  h_{\!P}(\xx_\eps) + f(\xx_\eps) \nonumber \\
	&\le& \lambda \, h_{\!P}(\xx_\eps) + f(\xx_\eps) ~\le~ f_* +\eps. \nonumber
	\EEA

\subsection{Proof of Lemma \ref{lem:subgrad_h}}

	Let $\xx, \xx \in \RR^n$. To establish Lipschitz continuity of $h_{\!P}$, we have
	\BEA
	& &|h_{\!P}(\xx)- h_{\!P}(\xx')| \nonumber \\
	&=& |P(d_{\C_1}\!(\xx),\ldots,d_{\C_m}\!(\xx))-P(d_{\C_1}\!(\xx'),\ldots,d_{\C_m}\!(\xx'))| \nonumber \\
	&\le& P(d_{\C_1}\!(\xx)-d_{\C_1}\!(\xx'),\ldots,d_{\C_m}\!(\xx)-d_{\C_m}\!(\xx')) \nonumber \\
	&=& P(|d_{\C_1}\!(\xx)-d_{\C_1}\!(\xx')|,\ldots,|d_{\C_m}\!(\xx)-d_{\C_m}\!(\xx')|) \nonumber \\
	&\le& P(\|\xx -\xx'\|,\ldots,\|\xx -\xx'\|)  \nonumber \\
	&=& \|\xx -\xx'\|\, P(\one),\nonumber
	\EEA
	where the first inequality is a result of triangle inequality of norm $P$ and the 2nd equality is due to $P$ being an absolute norm. Recall that  an absolute norm is monotonic, that is, the norm is monotonically non-decreasing in the absolute values of its components. Using this monotonicity property of $P$ together with \re{d_lip} results in the 2nd inequality above. Thus, $h_{\!P}$ is Lipschitz continuous on $\RR^n$ and $P(\one)$ is a Lipschitz constant.
	
	Using the property that dual norm of an absolute norm is also an absolute norm, we have $P_*$ as an absolute norm. Now, to find a subgradient we first present the following characterization of $h_{\!P}$ using dual norm $P_*$:
	\BEA
	h_{\!P}(\xx) &=& P(d_{\C_1}\!(\xx),\ldots,d_{\C_m}\!(\xx)) \nonumber \\
	&=& \max_{\uu\in \RR^m:\, P_*(\uu)\le 1}~\sum_{i=1}^m u_i d_{\C_i}\!(\xx) \nonumber \\
	&=& \max_{\uu\in \RR^m_+:\, P_*(\uu)\le 1}~\sum_{i=1}^m u_i d_{\C_i}\!(\xx), \label{eq:h_sub}
	\EEA 
	where the last equality follows because $d_{\C_i}(\xx) \ge 0\, \forall i$ and $P_*$ is an absolute norm. Since distance functions are convex, \re{h_sub} represents $h_{\!P}$ as maximum of a family of convex functions. Therefore, if $\uu^*$ is a maximizer of \re{h_sub} and $\xi_i$ denotes a subgradient of $d_{\C_i}$ at $\xx$, then $\sum_{i=1}^m u_i^*\xi_i$ is a subgradient of $h_{\!P}$ at $\xx$. Now, part 3 of Proposition~\ref{prop:dist} says we can use $\xi_i = \frac{\xx - \pr_{\C_i}(\xx)}{\| \xx - \pr_{\C_i}(\xx) \|}$. This completes the proof.

\subsection{Split-Projection subgradient (\SPS) Algorithm}
\begin{algorithm}\label{algo:subgradient}
		\caption{Split-Projection Subgradient (\SPS) Algorithm to solve \re{p0} when $f$ is nonsmooth}
		\begin{algorithmic}
			\STATE Input: $\lambda>0$, number of iterations $T$.
			\STATE  Initialization: ~$\xx^{(1)} \in \X$.
			\FOR {$t=1$ {\bfseries to} $T$}
			\STATE  Get subgradient $f'( \xx^{(t)})$ of $f$ at $\xx^{(t)}$.
			\STATE  Get projections: $\pr_{\C_i}\!(\xx^{(t)}),\, \ldots,\pr_{\C_m}\!(\xx^{(t)})$. 
			\STATE  Compute $h_{\!P}'(\xx^{(t)})$ using Lemma \ref{lem:subgrad_h}.
			\STATE  Set $\xi^{(t)} := f'( \xx^{(t)}) + \lambda \, h_{\!P}'(\xx^{(t)})$.
			\STATE  Choose step-size ~$\gamma_t > 0$.
			\STATE  Update $\xx^{(t+1)} := \pr_{\X}\!\left( \xx^{(t)} - \gamma_t \xi^{t}\right)$.
			\ENDFOR
			\STATE Output: $\widehat{\xx}^{(T)} ~\dfn ~ [\sum_{t=1}^T\gamma_t^{-1} \xx^{(t)} ]\, /\,[\sum_{i=1}^T\gamma_i^{-1}]$.
		\end{algorithmic}
\end{algorithm}

\subsection{Proof of Proposition \ref{prop:nonsmooth_convex}}

	We consider the \SPS~algorithm applied to problem \re{p0} with $\lambda \ge 2 \Upsilon_{\!\!P} \lip_f$. Let $D \dfn \max_{\xx,\xx'\in \X}\|\xx-\xx'\|$ be the diameter of the compact set $\X$. Let the stepsizes be chosen as $\gamma_t = \frac{\eta}{\sqrt{t}} , 1 \le t \le T$, for some $\eta >0$. Note that \SPS~algorithm for problem \re{p0} is nothing but application of standard subgradient algorithm to the exact penalty based reformulation \re{p1}. Therefore, we apply the convergence rate guarantee of standard subgradient method from Corollary~2 of \cite{Nedic_weight_avg} which provides the following bound on the output $\widehat{\xx}^{(T)}$:
	\BEQ f^{\lambda}(\widehat{\xx}^{(T)}) - f^{\lambda}_* ~\le~  \frac{3}{4\sqrt{T}}\left( \frac{D^2}{\eta}+ \eta [\lip_f+\lambda P(\one)]^2 \right), \label{eq:nonsmooth_bound}\EEQ
	where $[\lip_f+\lambda P(\one)]$ is a Lipschitz constant of $f^{\lambda} \equiv f + \lambda h_{\!P}$. By minimizing the right hand side of \re{nonsmooth_bound} we obtain optimal value of $\eta$ as $\frac{D}{\lip_f+\lambda P(\one)}$. Now, substituting the optimal value of $\eta$ in \re{nonsmooth_bound} we get:
	$$ f^{\lambda}(\widehat{\xx}^{(T)}) - f^{\lambda}_* \le  \frac{3 D [\lip_f+\lambda P(\one)] }{2\sqrt{T}}.$$
	Recall that $\lambda \ge 2 \Upsilon_{\!\!P}\lip_f$. Thus, to produce an $\eps$-optimal solution of \re{p1} we need $O\big(\frac{\lip_f^2D^2}{\eps^2}\big)$ iterations of the \SPS-algorithm. As per Proposition~\ref{prop:exact_penalty} such $\eps$-optimal solutions of \re{p1} are in fact $\eps$-optimal ${\eps}$-feasible solutions of \re{p0}. This completes the proof.

\subsection{Proof of Proposition \ref{prop:nonsmooth_strong}}

	For strongly convex case, we set stepsizes in the \SPS~algorithm as $\gamma_t = \frac{2}{\mu_{\!f}(t+1)} , 1 \le t \le T$. Recall that \SPS~algorithm is nothing but an instance of standard subgradient algorithm applied to \re{p1}. Now, we quote the following convergence rate guarantee of standard subgradient method from Corollary 1 of \cite{Nedic_weight_avg}:
	\BEQ f^{\lambda}(\widehat{\xx}^{(T)}) - f^{\lambda}_* ~\le~  \frac{2}{\mu_{\!f} T}[\lip_f+\lambda P(\one)]^2.\nonumber \EEQ
	Also, we choose $\lambda \ge 2 \Upsilon_{\!\!P}\lip_f$. Therefore, \SPS~algorithm produces an $\eps$-optimal solution of \re{p1} in $O\big(\frac{\lip_f^2}{\mu_{\!f}\eps}\big)$ iterations. Now, applying Proposition~\ref{prop:exact_penalty} we achieve the desired $O(1/\eps)$ complexity for an $\eps$-optimal $\eps$-feasible solution of \re{p0}.

\subsection{Proof of Lemma \ref{lem:hp_dual}}

	Applying Proposition \ref{prop:dist} to $\C_i$ we have $\forall \xx \in \RR^n:$
	\BEQ d_{\C_i}(\xx) =  \max_{\yy_i \in\RR^n:\|\yy_i \| \le 1 } [\xx\t\yy_i-\supp_{\A}(\yy_i)],\EEQ
	where the maxima is achieved at a point with $\|y_i\| = 1$.
	Using the above characterization of $d_{\C_i}$ in \re{h_sub} we get
	\BEA
	h_{\!P}(\xx)= \max_{\uu\in \RR^m_+:\, P_*(\uu)\le 1,\,\,\|\yy_i \| = 1 \,\forall i} \sum_{i=1}^m  u_i [\xx\t\yy_i-\supp_{\C_i}(\yy_i)]. \nonumber
	\EEA 
	Note that $\forall \,u_i \ge 0: ~u_i \supp_{\C_i}(\yy_i) = \supp_{\C_i}(\uu_i \yy_i)$. Therefore, making the variable transformation $u_i \yy_i \mapsto \tilde{\yy}_i$ we achieve the desired form:
	\BEA
	h_{\!P}(\xx)= \max_{\tilde{\yy}_i \in \RR^n \forall i:\, P_*\!(\|\tilde{\yy}_1\|,\ldots,\|\tilde{\yy}_m\|)\le 1}\, \sum_{i=1}^m [\xx\t\tilde{\yy}_i-\supp_{\C_i}(\tilde{\yy}_i)] . \nonumber
	\EEA

\subsection{Proof of Lemma \ref{lem:sad_sol}}

	From \re{p1}, \re{hp_dual} and \re{p2} we have $$f^{\lambda}(\xx) = \max_{\bY\in\Y^{\lambda}_P} \L(\xx,\bY )$$ and
	\BEA
	f^{\lambda}_* = \min_{\xx}f^{\lambda}(\xx) &=&  \min_{\xx \in \X} \max_{\bY\in\Y^{\lambda}_P} \L(\xx,\bY ) \nonumber \\
	&=& \max_{\bY\in\Y^{\lambda}_P} \min_{\xx \in \X} \L(\xx,\bY ) \nonumber \\
	&=& \max_{\bY\in\Y^{\lambda}_P} q(\bY) \label{eq:dual_sad},
	\EEA
	where we define $ q(\bY) \dfn \min_{\xx \in \X} \L(\xx,\bY )$ and swapping of $\min$ \& $\max$ holds due to Min-Max Theorem \cite{Minmax_Sion}. 
	We are given $(\xx_\eps,\bY_\eps) \in \X \times \Y^{\lambda}_P$ ~such that
	\BEA
	& &\sup_{\xx \in \X,\,\bY\in\Y^{\lambda}_P} ~[~\L(\xx_\eps,\bY ) - \L(\xx,\bY_\eps)~] ~\le~ \eps \nonumber \\
	& \Rightarrow & [ \sup_{\bY\in\Y^{\lambda}_P} \L(\xx_\eps,\bY ) - \inf_{\xx \in \X}\L(\xx,\bY_\eps)] ~\le~ \eps \nonumber \\
	& \Rightarrow & [ f^{\lambda}(\xx_\eps) - q(\bY_\eps) ] ~\le~ \eps \nonumber \\
	& \Rightarrow & [ f^{\lambda}(\xx_\eps) - f^{\lambda}_*]+[f^{\lambda}_* - q(\bY_\eps) ] ~\le~ \eps. \nonumber 
	\EEA
	Now, from \re{dual_sad} we have $ q(\bY_\eps) \le \max_{\bY\in\Y^{\lambda}_P} q(\bY) = f^{\lambda}_*$. Thus, we have
	$f^{\lambda}(\xx_\eps) - f^{\lambda}_* \le \eps$. Therefore, $\xx_\eps$ is an $\eps$-optimal solution of \re{p1}. Now, by virtue of part-c of Proposition~\ref{prop:exact_penalty} $\xx_\eps$ is an $\eps$-optimal $\eps$-feasible solution of \re{p0}.

\subsection{Proof of Lemma \ref{lem:prox_g}}
	By definition of proximal operator, we have 
	\BEQ \prox_{\gamma g}(\bY) = \argmin_{\bY' \in \Y^{\lambda}_P}\sum_{i=1}^m\left[\frac{1}{2\gamma}\| \yy_i - \yy_i'\|^2 + \supp_{\C_i}(\yy_i')\right], \nonumber \EEQ
	where $\Y^{\lambda}_P \dfn \{ \bY \in  \otimes_{i=1}^m \RR^n \,|\,P_*(\|\yy_1\|,\ldots,\|\yy_m\|) \le \lambda\}$. Utilizing monotonicity of norm $P_*$, we break the above minimization problem as follows: 
	\BEQ %\min_{\uu \in \RR^m_+:P_*(\uu)\le 1} \min_{\|\yy_i'\|\le u_i}\sum_{i=1}^m\left[\frac{1}{2\gamma}\| \yy_i - \yy_i'\|^2 + \supp_{\C_i}(\yy_i')\right] \nonumber \\
	%\min_{\uu \in \RR^m_+:\,P_*(\uu)\le 1} \sum_{i=1}^m \min_{\|\yy_i'\|\le u_i}\left[\frac{1}{2\gamma}\| \yy_i - \yy_i'\|^2 + \supp_{\C_i}(\yy_i')\right], 	
	\min_{\uu \in \RR^m_+:\,P_*(\uu)\le 1} \sum_{i=1}^m \left[\frac{1}{2\gamma}\| \yy_i - \yy_i^*\|^2 + \supp_{\C_i}(\yy_i^*)\right], \label{eq:uy1} \EEQ
	where for a fixed $\uu=(u_1,\ldots,u_m) \in \RR^m_+$ we find $\yy_i^*$ as
	\BEA
	\yy_i^* = \argmin_{\yy_i' \in \RR^n : \|\yy_i'\|\le \eta_i } ~\frac{1}{2\gamma}\| \yy_i - \yy_i'\|^2 + \supp_{\C_i}(\yy_i') \nonumber \\
	=  \argmin_{\yy_i' \in \RR^n : \|\yy_i'\|\le u_i } \max_{\xx_i \in \C_i } ~\frac{1}{2\gamma}\| \yy_i - \yy_i'\|^2 + \xx\t_{i} \yy_i' \nonumber \\
	=  \max_{\xx_i \in \C_i } \argmin_{\yy_i' \in \RR^n : \|\yy_i'\|\le u_i } ~\frac{1}{2\gamma}\| \yy_i - \yy_i'\|^2 +\xx\t_{i} \yy_i' \label{eq:xy1}
	\EEA
	Now, for a fixed $\xx_i$, the inner minimization w.r.t. $\yy_i'$ is achieved at
	$$\yy_i'^* = \min\left\{1, \frac{u_i}{\|\yy_i - \gamma \xx_i\|}\right\} [\yy_i - \gamma \xx_i].$$ 
	Substituting $\yy_i'$ with $\yy_i'^*$ in \re{xy1} we now find the value of $\xx_i$ as 
	\BEA
	 & & \argmax_{\xx_i \in \C_i }\,\frac{1}{2\gamma}\| \yy_i - \yy_i'^*\|^2 +\xx\t_{i} \yy_i^* \nonumber \\
	 &=& \argmax_{\xx_i \in \C_i }\,\| (\yy_i-\gamma {\xx_i}) - \yy_i'^*\|^2 - \| (\yy_i-\gamma {\xx_i})\|^2 \nonumber \\
	 &=& \argmin_{\xx_i \in \C_i } \,\|\yy_i-\gamma {\xx_i}\|^2 (1-[1-\min\{1, \frac{u_i}{\|\yy_i - \gamma \xx_i\|}\}]^2). \nonumber% \\
%	&=& \argmin_{\xx_i \in \C_i } \,\|\yy_i-\gamma {\xx_i}\| = \pr_{\C_i}(\gamma^{-1}\yy_i).
	 \EEA
	The last minimization actually boils down to $\argmin_{\xx_i \in \C_i } \,\|\yy_i-\gamma {\xx_i}\|$ which is achieved at $ \xx_i = \pr_{\C_i}(\gamma^{-1}\yy_i)$. Now substituting this value of $\xx_i$ in $\yy_i'^*$ and 
	defining $\hat{\yy}_i \dfn \yy_i - \gamma \pr_{\C_i}(\gamma^{-1}\yy_i)$ we have
	 \BEQ
	\yy_i^* %&=& \min\{1, \frac{\eta_i}{\|\yy_i - \gamma \pr_{\C_i}(\gamma^{-1}\yy_i)\|}\} [\yy_i - \gamma \pr_{\C_i}(\gamma^{-1}\yy_i)] \nonumber \\
	= \min\left\{u_i,\|\hat{\yy}_i\|\right\}\,\frac{\hat{\yy}_i}{{\|\hat{\yy}_i\|}} \nonumber.
	\EEQ
	This bring us to the problem of finding optimal $u_i$ by solving \re{uy1} with $y_i^*$ as above. Since, $P_*$ is a monotonic norm we can equivalently solve the following:
	\BEQ 
		\min_{\etab \in \RR^m_+:\,P_*(\etab)\le 1} \sum_{i=1}^m \left[\frac{1}{2\gamma}\| \yy_i - \yy_i^*\|^2 + \supp_{\C_i}(\yy_i^*)\right], \label{eq:yy1}
	\EEQ
	where $\eta_i = \min\{ u_i,\|\yy_i\|\} $, $\yy_i^* =  \frac{\eta_i}{\|\hat{\yy}_i\|}\hat{\yy}_i$ and $\hat{\yy}_i = \yy_i - \gamma \pr_{\C_i}(\gamma^{-1}\yy_i)$. Using the definition of $\pr_{\C_i}(\gamma^{-1}\yy_i)$ we find $\supp_{\C_i}(\yy_i^*) = \pr_{\C_i}(\gamma^{-1}\yy_i)\t \yy_i^*$ which we substitute in \re{yy1}. Therefore \re{yy1} becomes
	\BEA
	& &\argmin_{\etab \in \RR^m_+:\,P_*(\etab)\le 1} \sum_{i=1}^m \left[\frac{1}{2\gamma}\| \yy_i - \yy_i^*\|^2 +   \pr_{\C_i}(\gamma^{-1}\yy_i)\t \yy_i^* \right] \nonumber \\
	&=& \argmin_{\etab \in \RR^m_+:\,P_*(\etab)\le 1} \sum_{i=1}^m \| \hat{\yy}_i - \yy_i^*\|^2 \nonumber	\\
	&=& \argmin_{\etab \in \RR^m_+:\,P_*(\etab)\le 1} \sum_{i=1}^m \left\| \hat{\yy}_i - \frac{\eta_i}{\|\hat{\yy}_i\|}\hat{\yy}_i \right\|^2 \nonumber \\
	&=& \argmin_{\etab \in \RR^m_+:\,P_*(\etab)\le 1} \sum_{i=1}^m [\,\eta_i - \|\hat{\yy}_i\|\, ]^2. \nonumber \\
	&=& \argmin_{\etab \in \RR^m:\,P_*(\etab)\le 1} \sum_{i=1}^m [\,\eta_i - \|\hat{\yy}_i\|\, ]^2. \nonumber 
	\EEA	
	where the last equality holds as $P_*$ is an absolute norm. Thus, the optimal $\etab$ is the projection of $( \|\hat{\yy}_1\|, \ldots, \|\hat{\yy}_m\|)$ onto $\{\etab \in \RR^m:\,P_*(\etab)\le 1\}$. Let $(r_1,\ldots,r_m)$ be the projection. Therefore, substituting the optimal value of $\eta_i$ in $\yy_i^*  =  \frac{\eta_i}{\|\hat{\yy}_i\|}\hat{\yy}_i$ we achieve the desired result.

\subsection{Exact Penalty based Primal Dual (\EPPD) Algorithm}

Before stating the algorithm we recall the following notation: $\bY \dfn (\yy_1,\ldots,\yy_m) \in \otimes_{i=1}^m \RR^n$ and the operator $\bA$ maps $\xx$ to $(\xx,\ldots,\xx) \in \otimes_{i=1}^m \RR^n$. Let $M_{\!f} > 0$ be the Lipschitz constant of the gradient of $f$ and $D$ be the diameter of the set $\X$. We set the stepsize parameters $\tau,\,\gamma$ in the algorithm as follows:
$$ \gamma = \frac{\lambda} {D}, ~\tau = \frac{1}{M_{\!f} + m \gamma} .$$

\begin{algorithm}\label{algo:eppd}
	\caption{Exact Penalty based Primal Dual (\EPPD) Algorithm to solve \re{p0} when $f$ is smooth }
	\begin{algorithmic}
		\STATE Input: $\lambda>0$, number of iterations $T$.
		\STATE Initialization: ~$\xx^{(0)} \in \X,\,\bY^{(0)} = 0 $.
		\STATE Choose stepsize parameters $\tau,\gamma >0$.
		\FOR {$t=0$ {\bfseries to} $T-1$}
		\STATE $\xx^{(t+1)} := \pr_{\X}\!\left( \xx^{(t)} - \tau[ \grad\!f(\xx^{(t)}) + \sum_{i=1}^m \yy_i^{(t)} ] \right).$
		\STATE $\bY^{(t+1)} := \prox_{\gamma g}\!\left( \bY^{(t)}+\gamma \bA [2\xx^{(t+1)} - \xx^{(t)}] \right).$
		\ENDFOR
		\STATE Output: $\widehat{\xx}^{(T)} \dfn \frac{1}{T}\sum_{t=1}^T \xx^{(t)}$.
	\end{algorithmic}
\end{algorithm}

\subsection{Exact Penalty based Accelerated Primal Dual (EPAPD) Algorithm}

\begin{algorithm}\label{algo:eppd_strong}
	\caption{Exact Penalty  based Accelerated Primal Dual (\EPAPD) Algorithm to solve \re{p0} when $f$ is smooth and strongly convex}
	\begin{algorithmic}
		\STATE Input: $\lambda>0$, number of iterations $T$.
		\STATE Initialization: ~$\xx^{(0)} \in \X,\,\bY^{(0)} = 0, \xx^{(-1)} = \xx^{(0)}$.
		\FOR {$t=0$ {\bfseries to} $T-1$}
		\STATE Choose parameters $\tau_t,\,\gamma_t,\theta_t. $
		\STATE $\bY^{(t+1)} := \prox_{\gamma g}\!\left( \bY^{(t)}+\gamma_t \bA [\xx^{(t)} + \theta_t( \xx^{(t)} -\xx^{(t-1)})] \right).$
		\STATE $\xx^{(t+1)} := \pr_{\X}\!\left( \xx^{(t)} - \tau_t [ \grad\!f(\xx^{(t)}) + \sum_{i=1}^m \yy_i^{(t+1)} ] \right).$
		\ENDFOR
		\STATE Output: $\widehat{\xx}^{(T)} \dfn \frac{1}{T}\sum_{t=1}^T \xx^{(t)}$.
	\end{algorithmic}
\end{algorithm}
Let $\mu_{\!f} >0$ be the modulus of strong convexity of $f$ and $M_{\!f} > 0$ be the Lipschitz constant of the gradient of $f$. As in \cite{Pock_2016_ergodic} we choose the algorithm parameters $\tau_t,\,\gamma_t,\theta_t $ based on the following recursions:
\BEA
\theta_0 &=& 1, ~ \tau_0 = \frac{1}{2M_{\!f}},~ \gamma_0 = \frac{M_{\!f}}{m},\nonumber  \\
\theta_{t+1} &:=& \frac{1}{\sqrt{1+\mu\tau_t}}, \nonumber \\
\tau_{t+1} &:=& \theta_{t+1} \tau_t , \nonumber \\
\gamma_{t+1} &:=&  \gamma_t / \theta_{t+1}. \nonumber
\EEA

\subsection{Split Mirror Prox (\MPICS) Algorithm}
\begin{algorithm}\label{algo:smp}
	\caption{Split Mirror Prox (\MPICS) Algorithm to solve \re{p0} when $f$ is given by \re{spf} }
	\begin{algorithmic}
		\STATE Input: $\lambda>0$, number of iterations $T$.
		\STATE Initialization: ~$\xx^{(1)} \in \X,\,\zz^{(1)} \in \Z,\,\bY^{(1)} = 0$.
		\STATE Choose stepsizes $\gamma_x,\,\gamma_y,\,\gamma_z >0$.
		\FOR {$t=1$ {\bfseries to} $T$}
		\STATE $\bullet$~$\tilde{\xx}^{(t)} := \pr_{\X}\!\left( \xx^{(t)} - \gamma_x[ \grad_{\!\xx}F( \xx^{(t)},\zz^{(t)})+\sum_{i=1}^m\yy_i^{(t)} ] \right).$
		\STATE $\bullet$~$\tilde{\zz}^{(t)} := \pr_{\Z}\!\left( \zz^{(t)} + \gamma_z \grad_{\!\zz}F( \xx^{(t)},\zz^{(t)}) \right).$
		\STATE $\bullet$~$\tilde{\bY}^{(t)} := \prox_{\gamma_y g}\!\left( \bY^{(t)}+\gamma_y \bA\xx^{(t)} \right).$
		\STATE $\bullet$~$g'(\tilde{\bY}^{(t)} ) := \left( \pr_{\C_1}\!(\frac{\yy^{(t)}_1}{\gamma_y}),\ldots,\, \pr_{\C_m}\!(\frac{\yy^{(t)}_m}{\gamma_y})\right).$
		\STATE $\bullet$~$\xx^{(t+1)} := \pr_{\X}\!\left( \xx^{(t)} - \gamma_x[ \grad_{\!\xx}F( \tilde{\xx}^{(t)},\tilde{\zz}^{(t)} ) + \sum_{i=1}^m\tilde{\yy}_i^{(t)} ] \right).$
		\STATE $\bullet$~$\zz^{(t+1)} := \pr_{\Z}\!\left( \zz^{(t)} + \gamma_z \grad_{\!\zz}F(\tilde{\xx}^{(t)},\tilde{\zz}^{(t)} ) \right).$
		\STATE $\bullet$~$\bY^{(t+1)} := \pr_{\Y^{\lambda}_{\!P}}\!\left( \bY^{(t)}+\gamma_y [\bA\tilde{\xx}^{(t)} - g'(\tilde{\bY}^{(t)} )]\right).$
		\ENDFOR
		\STATE Output: $\widehat{\xx}^{(T)} \dfn \frac{1}{T}\sum_{t=1}^T \tilde{\xx}^{(t)}$.
	\end{algorithmic}
\end{algorithm}
We recall from Lemma~\ref{lem:prox_g} that proximal operator of $g$ is computed through projections onto the sets $\C_1,\ldots,\,\C_m$. The same projections were used to construct $g'(\tilde{\bY}^{(t)} )$ in the above algorithm. Hence, every iteration of \MPICS~algorithm requires projecting onto each $\C_i$'s only once. Also, the last projection onto the set $\Y^{\lambda}_{\!P}$ can be computed as follows: $\pr_{\Y^{\lambda}_{\!P}}(\bY) =  (r_1{\yy}_1/\|{\yy}_1\|,\,\ldots,\,r_m{\yy}_m/\|{\yy}_m\| )$, where $(r_1,\,\ldots,r_m)$ is the projection of $(\|\yy_1\|,\,\ldots,\,\|\yy_m\|)$ onto $\{\uu\in\RR^m\,|\,P_*(\uu)\le \lambda \}$. As done in standard mirror-prox \cite{Nemirovski_MLopt2} we set the stepsize parameters $\gamma_x,\,\gamma_y,\,\gamma_z$ based on the diameters of the sets $\X,\Z,\Y^{\lambda}_{\!P}$ and the Lipschitz constant of $F$.

\subsection{Proof of Proposition \ref{prop:MPICS}}

From \cite{Nemirovski_MLopt2} we have that iteration complexity of the Mirror Prox-a (MPa) Algorithm is $O( {1}/{\eps})$. Note that our \MPICS~algorithm is nothing but standard MPa algorithm applied to problem \re{p3}. As \re{p3} is the saddle-point version of the primal problem \re{p1}, \MPICS~ãlgorithm will produce an $\eps$-optimal solution of \re{p1} in $O( {1}/{\eps})$ iterations. Now apply Proposition \ref{prop:exact_penalty} which ensures that $\eps$-optimal solution of \re{p1} is an $\eps$-optimal $\eps$-feasible solution of the original problem \re{p0}.

\subsection{Dealing with unknown Regularity Constant}	
	We note from Proposition~\ref{prop:exact_penalty} that setting $\lambda \ge 2 \Upsilon_{\!\!P}\lip_f$ ensures that an $\eps$-optimal solution of \re{p1} is also an $\eps$-optimal $\eps$-feasible solution of the original problem \re{p0}. Here we state an algorithmic strategy to deal with the case when the regularity constant $\Upsilon_{\!\!P}$ is not known. Proposed methods find an $\eps$-optimal $\eps$-feasible solution of \re{p0} by solving \re{p1}. We fix an $\eps>0$ and consider a first-order method $M$ for obtaining an $\eps$-optimal solution of \re{p1}. Let the number of iterations required be $T_{\lambda} = C{\lambda^a}/{\eps^b},$ where $C,a,b$ are positive constants. When the regularity constant $\Upsilon_{\!\!P}$ is not known we start with $\lambda = \lambda_0$ for some $\lambda_0 >0$ and run $T_{\lambda_0}$ iterations of $M$. If the output after $T_{\lambda_0}$ iterations satisfy the $\eps$-feasibility then we are done; otherwise we double the value of $\lambda$ and run $T_{\lambda}$ iterations of $M$ with the new value of $\lambda$. If we proceed in this way at one point we will have $\lambda \ge 2 \Upsilon_{\!\!P}\lip_f$ and the algorithm will stop with an $\eps$-optimal $\eps$-feasible solution of \re{p0}. It is easy to see that the total number of iterations required to finally stop is only a constant factor times the number of iterations required if the regularity constant was known. Hence, we achieve the same complexity of $O(1/\eps^b)$.

\end{document}